\documentclass[11pt]{amsart}
\usepackage[sc]{mathpazo}
\usepackage[T1]{fontenc}
\usepackage{latexsym,amscd,amssymb, graphicx, color, amsthm, bm, amsmath, cancel, enumitem}  
\usepackage{mathtools}
\usepackage{tikz}
\usepackage[all,cmtip]{xy}

\usepackage[hidelinks,backref=page]{hyperref} 
 \hypersetup{
    colorlinks,
    citecolor=red,
    filecolor=magenta,
    linkcolor=blue,
    urlcolor=black
}

\theoremstyle{plain}
\newtheorem{theorem}{Theorem}[section]

\newtheorem{lemma}[theorem]{Lemma}
\newtheorem{corollary}[theorem]{Corollary}
\newtheorem{conjecture}[theorem]{Conjecture}

\theoremstyle{definition}
\newtheorem{definition}[theorem]{Definition}
\newtheorem{example}[theorem]{Example}

\theoremstyle{remark}
\newtheorem{remark}[theorem]{Remark}

\numberwithin{equation}{section}

\newcommand{\Hilb}{{\mathrm{Hilb}}}

\newcommand{\symm}{{\mathfrak{S}}}

\newcommand{\AAA}{{\mathcal{A}}}
\newcommand{\BBB}{{\mathcal{B}}}

\newcommand{\MMM}{\mathfrak{M}} 
\newcommand{\CC}{{\mathbb{C}}}
\newcommand{\RR}{{\mathbb{R}}}

\newcommand{\braid}{\mathrm{braid}}

\newcommand{\xx}{{\bm{x}}}

\newcommand{\ttheta}{{{\bm \theta}}}

\newcommand{\mc}[1]{\mathcal{#1}} 
\newcommand{\rank}[1]{\mathsf{rk}(#1)}

\newcommand{\intact}{\operatorname{ia}}
\newcommand{\extact}{\operatorname{ea}}
\newcommand{\intpas}{\operatorname{ip}}
\newcommand{\extpas}{\operatorname{ep}}
\newcommand{\hilb}{\operatorname{Hilb}}

\newcommand{\suchthat}{\;|\;}
\newcommand{\external}{\mathcal{E}} 
\newcommand{\groundset}{\mathsf{E}} 

\author{Brendon Rhoades}
\address{Department of Mathematics,
University of California, San Diego,
La Jolla, CA, 92093, USA}
\email{\href{mailto:bprhoades@ucsd.edu}{bprhoades@ucsd.edu}}

\author{Vasu Tewari}
\address{Department of Mathematical and Computational Sciences, University of Toronto Mississauga, Mississauga, ON L5L 1C6, Canada}
\email{\href{mailto:vasu.tewari@utoronto.ca}{vasu.tewari@utoronto.ca}}

\author{Andy Wilson}
\address{Department of Mathematics,
Kennesaw State University,
Marietta, GA, 30060, USA}
\email{\href{mailto:awils342@kennesaw.edu}{awils342@kennesaw.edu}}

\thanks{
BR was partially supported by NSF Grant DMS-2246846. VT was partially supported by NSF Grant DMS-2246961. AW was partially supported by AMS--Simons PUI Grant 434651.}


\begin{document}

\title[]
{Tutte polynomials in superspace}

\begin{abstract}
We associate a quotient of superspace to any hyperplane arrangement by considering the differential closure of an ideal generated by powers of certain homogeneous linear forms.
This quotient is a superspace analogue of the external zonotopal algebra, and it further contains the central zonotopal algebra in the appropriate grading.
We show that an evaluation of the  bivariate Tutte polynomial is the bigraded Hilbert series of this quotient. We then use this fact to construct an explicit basis for the Macaulay inverse.
These results generalize those of Ardila--Postnikov and Holtz--Ron.
We also discuss enumerative consequences of our results in the setting of hyperplane arrangements.
\end{abstract}

\maketitle

\section{Introduction}
\label{sec:intro}

Consider commutative generators $\xx_n = (x_1, \dots, x_n)$
and \emph{anticommutative} generators $\ttheta_n = (\theta_1, \dots, \theta_n)$.
We define \emph{rank $n$ superspace $\Omega_n$} to be $\CC[\xx_n] \otimes \wedge \{ \ttheta_n \}$.
This ring arises in physics \cite{PS95} where the two sets of variables represent states of bosons and fermions, respectively. This explains why the variables are often referred to as \emph{bosonic} and \emph{fermionic} variables.
Additionally, $\Omega_n$ is the Hochschild homology of the polynomial ring $\CC[\xx_n]$ and, as such, may be considered as the ring of polynomial-valued holomorphic differential forms on $\CC^n$. In this latter context, one interprets the fermionic variable $\theta_i$ as the differential $dx_i$.

In the last few years, inspired by the study of various generalizations of the coinvariant algebra from a representation theoretic viewpoint, there has been great interest in `superspace quotients.' 
So far, the primary impetus has come from Macdonald-theoretic considerations. In particular, a conjecture of Zabrocki \cite{zabrocki2019module} makes a fascinating connection between the `Delta conjecture' of Haglund--Remmel--Wilson \cite{HRW18} and a module that generalizes the diagonal coinvariants by introducing fermionic variables into the fray.
The influence of these works may be witnessed by the rapid activity in this area; see \cite{BCSZ23, iraci2024smirnov, IRR23,  Kim23,KRR23,RW22,rhoades2023hilbert,WS21,WS23} for a far-from-complete list.
We note also that there is another interesting strand of research \cite{BfDLM12,DLM01,DLM03,DLM06} studying lifts of classical symmetric function objects to superspace which predates recent work in this area, though we shall not take this perspective in this article.

In contrast to the aforementioned works,  we approach superspace from the perspective of hyperplane arrangements (or equivalently, realizable matroids).
Our point of entry is the theory of power ideals, as developed by Ardila--Postnikov \cite{AP10} and Holtz--Ron \cite{HR11}, amongst others; see \cite{CdPbook} for a book-length treatment on a large swathe of mathematics underlying the general area.
As we hope to demonstrate, power ideals have a superspace generalization that merits deeper investigation. In this article, we develop a complete understanding in what we shall refer to as the `superspace external' case, generalizing the classical external zonotopal algebra \cite{HR11}.

Regarding $\Omega_n$ as a ring
of differential forms, we have the {\em Euler operator}
(or {\em total derivative})
$d: \Omega_n \rightarrow \Omega_n$ defined by
\begin{equation}
    df \coloneqq \sum_{i = 1}^n \left( \partial f / \partial x_i \right) \cdot \theta_i
\end{equation}
where in the evaluation of 
$\left( \partial f / \partial x_i \right)$
the $\ttheta$-variables are treated as constants.
A {\em differential ideal} $I \subseteq \Omega_n$ 
is an ideal closed under the action of $d$.

Throughout we treat the $\xx$-variables and $\ttheta$-variables as elements of degree $(1,0)$ and $(0,1)$, respectively.
Let $I \subseteq \CC[\xx_n]$ be a homogeneous ideal.
The {\em differential closure} 
$I^\theta \subseteq \Omega_n$ of $I$ is the smallest 
differential ideal in $\Omega_n$ containing $I$.
If $g_1, \dots, g_r \in \CC[\xx_n]$ are homogeneous 
generators of $I$, it is not hard to see\footnote{using
the relation $d \circ d = 0$ together
with the product rule
$d(f \cdot g) = d f \cdot g \pm f \cdot dg$
for bihomogeneous $f, g \in \Omega_n$} 
that
$I^\theta$ has generating set 
\begin{equation}
I^{\theta}=
(g_1, \dots, g_r, d g_1, \dots, d g_r) \subseteq 
\Omega_n.
\end{equation}
As the ideal $I^{\theta} \subseteq \Omega_n$ 
is bihomogeneous, 
the quotient space $\Omega_n/I^{\theta}$ 
acquires a bigrading wherein the $\theta$-degree $0$ 
piece is the original graded quotient $\CC[\xx_n]/I$. 
If the commutative quotient $\CC[\xx_n]/I$ has interesting
combinatorial, algebraic, or geometric properties,
it is natural to ask whether and how these properties
extend to the supercommutative quotient
$\Omega_n/I^\theta$.

We apply this recipe to power ideals determined by hyperplane arrangements.
Recall that a linear hyperplane in $\CC^n$ is a codimension $1$ subspace. More generally,
a hyperplane is any affine translation of a linear hyperplane.
We will be interested in \emph{multiarrangements} $\AAA = \{ H_1, \dots, H_m \}$ of hyperplanes in $\CC^n$, i.e. collections of hyperplanes wherein we allow repeats.
Given a multiarrangement $\AAA$, we refer to the number of hyperplanes taken with multiplicity as its \emph{size}.
Associated with $\AAA$ is a family of homogeneous ideals $J_{\AAA,k}\subseteq \CC[\xx_n]$ for $k\geq -1$ with the property that the generators are certain powers of homogeneous linear forms. 
The resulting quotient rings have fascinating mathematical properties with deep ties to numerical analysis, algebra, geometry, and combinatorics, particularly when $k\in \{-1,0,1\}$ \cite{Be18,DM85,Lenz12, Lenz16, Nen19,PS04, RRT23} . 
See Subsection \ref{subsec:superpower} for a precise definition of these ideals. 

In this article, we focus on the differential closure of $J_{\AAA,1}$ inside $\Omega_n$, which we denote $I_{\AAA}$, and the resulting quotient space $\external_{\AAA} = \Omega_n/I_\AAA$.
This space may be considered as a lift of the classical external zonotopal algebra.
Our first main result is the description of the bigraded Hilbert series for the superspace quotient $\external_{\AAA}$ as a specialization of the Tutte polynomial. 
This generalizes results of Ardila--Postnikov \cite{AP10} and Holtz--Ron \cite{HR11} wherein a single parameter specialization of the Tutte polynomial was obtained as a singly-graded Hilbert series.
It bears emphasizing that the full bivariate Tutte polynomial (up to reparametrization) has indeed been obtained as Hilbert series in related contexts. 
Berget \cite[Theorem 1.1]{Ber10} constructed a bigraded \emph{vector subspace} of the polynomial ring whose Hilbert series is shown to be a two-parameter Tutte evaluation; see also \cite{De01} for a related construction. As further explained in \cite{Ber10}, Terao \cite{Te02}, building on earlier work of Orlik--Terao \cite{OT94}, constructed a \emph{commutative} bigraded algebra with a similar property. 
In contrast, our construction (also a bigraded algebra) involves $\Omega_n$ which is the tensor product of the symmetric algebra and the exterior algebra on an $n$-dimensional space; whereas the bigradings
in \cite{Ber10,OT94} depend on the arrangement in 
question, our bigrading is induced by the natural
bosonic/fermionic on $\Omega_n$ and 
is independent of the arrangement.
This aspect amongst others distinguishes our construction from the ones just mentioned.

\begin{theorem}\label{th:main_1}
For any rank $r$ multiarrangement $\AAA$ of size $m$ in $\CC^n$ we have the bigraded Hilbert series
     \[
        \Hilb\left (\external_{\AAA}; q,t\right ) = (1+t)^rq^{m-r}T_{\AAA}\left (\frac{1+q+t}{1+t},\frac{1}{q}\right ).
     \]    
where $q$ tracks bosonic degree and $t$ tracks
fermionic degree.
\end{theorem}
The proof of Theorem~\ref{th:main_1}
involves a short exact sequence (see~\eqref{exact-sequence}) that witnesses a deletion-restriction type recursion. 
While deletion-restriction arguments are  
typical when dealing with evaluations of Tutte polynomials
and short exact sequences have arisen
in the commutative context
(see e.g. \cite[Proof of Prop. 4.5]{AP10}),
the novelty of our approach lies in using fermionic 
generators in a manner that might not be obvious; 
see~\eqref{eq:non_standard}.

The equality in Theorem~\ref{th:main_1} has some pleasant consequences.
The first is that the superspace analogue of the external zonotopal algebra `interpolates' between the classical external zonotopal algebra and the classical central zonotopal algebra. Additionally, the specialization $q=1$ produces the face-vector of the polyhedral complex naturally produced by a generic hyperplane arrangement obtained by perturbing the linear multiarrangement $\AAA$;
see Remark~\ref{rmk:fields} 
for an analogous result \cite{rhoades2023hilbert}
in coinvariant theory where 
superization of an ideal yields information about 
faces of higher codimension. 
We are further able to show that the dimension of $\external_{\AAA}$ equals the number of regions in the generic deformation of the `doubled' arrangement obtained by adding to $\AAA$ a copy of each hyperplane in it. 
When restricted to the class of subarrangements of the braid arrangement called graphical arrangements, the preceding observation gives an algebraic interpretation to region counts of generic bigraphical arrangements studied by Hopkins and Perkinson \cite[Section 3]{HoPe16}.

Our second main result amounts to unraveling the Hilbert series in Theorem~\ref{th:main_1} in order to obtain an explicit description for a basis of the Macaulay inverse space $I_{\AAA}^{\perp}$, which is isomorphic to $\external_{\AAA}$ as vector spaces.
Let $\mc{B}$ denote the set of bases for the matroid $\MMM_{\AAA}$ corresponding to $\AAA$. 
Given $B\in \mc{B}$, we let $EP_\AAA(B)$, 
 $IP_\AAA(B)$, $IA_\AAA(B)$ denote respectively the sets of externally passive elements, internally passive elements, and internally active elements. These notions are defined in Section \ref{sec:basis}.
\begin{theorem}\label{th:main_2}
For $H\in \AAA$, let $\alpha_H$ denote the\footnote{unique up to a nonzero scalar} homogeneous linear form with zero set $H$.
The following set forms a basis for $I_{\AAA}^{\perp}$:
\begin{equation}
M_{\AAA}\coloneqq \bigcup_{B\in \mc{B}}
\left\{
\prod_{e \, \in \, E} \alpha_e \times \prod_{i \, \in \, I} d \alpha_i \times \prod_{s \, \in \, S} \alpha_s \times \prod_{t \, \in \, T} d \alpha_t \,:\, 
\begin{array}{c}
E = EP_\AAA(B), \, \,
I \subseteq IP_\AAA(B), \\ S, T \subseteq IA_\AAA(B), \, \, S \cap T = \varnothing
\end{array} 
\right\}.
\end{equation}
\end{theorem}
The notation $M_\AAA$ reflects the status of these elements as monomials
in $\alpha_H$ and $d \alpha_H$ for hyperplanes $H \in \AAA$.
Again we are able to recover known results (cf. cases $k=0,-1$ in  \cite[Proposition 4.5]{AP10}) by picking appropriate subsets of $M_{\AAA}$.

\medskip

\noindent\textbf{Outline:} We begin by introducing the necessary background in Section~\ref{sec:background}. 
In Section~\ref{sec:spanning} we define
our primary object of study, superpower ideals, 
and exhibit a spanning set for their Macaulay inverse
systems.
The arguments in this section are technical but yield a short exact sequence~\eqref{exact-sequence} that is foundational for our needs.
In Section~\ref{sec:payoff}, we use this exact sequence to obtain the bigraded Hilbert series of our superspace quotients as an evaluation of the Tutte polynomial of the underlying matroid. In Subsection~\ref{subsec:real}, we describe enumerative applications in the setting of real arrangements.
In Section~\ref{sec:basis}, we obtain an explicit basis for the Macaulay inverse by combinatorially interpreting the Hilbert series obtained earlier.
We close with some final remarks and~conjectures.

\section{Background}
\label{sec:background}
We shall generously, and often interchangeably, use notation and terminology from the setting of hyperplane multiarrangements and matroids without proper introduction. The reader is referred to standard resources \cite{OT92, Ox11, Zas75} on these topics for any unexplained jargon encountered.

\subsection{Superspace}
\label{subsec:superspace}
As mentioned in the introduction, superspace $\Omega_n$ is the tensor product of the symmetric algebra $\CC[\xx_n]$ in $n$ bosonic variables with the exterior algebra $\wedge\{\ttheta_n\}$ in $n$ fermionic variables. 
A \emph{monomial} in $\Omega_n$ is defined to be a product of a monomial in the $\xx$-variables with a nonzero monomial in the $\ttheta$-variables.
Monomials in the $\ttheta$-variables, i.e. fermionic monomials, are indexed up to sign by subsets $J=\{j_1<\cdots<j_r\}\subseteq [n]$. 
Given such a $J$ we set $\theta_J\coloneqq \theta_{j_1}\cdots \theta_{j_r}$.
If $x_1^{a_1}\cdots x_n^{a_n}\theta_{j_1}\cdots \theta_{j_k}$ is a monomial in $ \Omega_n$, we refer to the degree $a_1+\cdots+a_n$ in the $\xx$-variables as the \emph{bosonic} degree and the degree $k$ in the $\ttheta$-variables as the \emph{fermionic} degree.
The $\CC$-algebra $\Omega_n$ admits a direct sum 
decomposition
\begin{equation}
    \Omega_n = \bigoplus_{i, j \, \geq \, 0}
    (\Omega_n)_{i,j}
\end{equation}
where $(\Omega_n)_{i,j} = \CC[\xx_n]_i \otimes \wedge^j \{\ttheta_n\}$ 
consists of bihomogeneous 
elements of bosonic degree $i$ and fermionic degree $j$.

For $1 \leq i \leq n$, the usual partial derivative $\partial_i\coloneqq \partial/\partial x_i$ acts on the first tensor factor of $\Omega_n$ while the \emph{contraction operator}
 $\partial_i^{\theta}\coloneqq \partial/\partial \theta_i$ acts on the second factor by extending its action on \emph{fermionic} monomials:
 given distinct indices $1\leq j_1, \cdots, j_r\leq n$ we let
\begin{equation}
\partial_i^{\theta}(\theta_{j_1}\cdots \theta_{j_r})=
\left\lbrace
\begin{array}{cl}
(-1)^{s-1}\theta_{j_1}\cdots \widehat{\theta_{j_s}}\cdots \theta_{j_r} & j_s=i \text{ for some $1\leq s\leq r$}\\
0 & \text{otherwise}
\end{array}\right.
\end{equation} 
where the hat denotes omission.
The operators $\partial_i$ and $\partial_i^\theta$
satisfy the same relations as $\Omega_n$, i.e.
\begin{equation}
    \partial_i \partial_j = \partial_j \partial_i 
    \quad \quad
    \partial_i \partial^\theta_j = 
    \partial^\theta_j \partial_i \quad \quad
    \partial^\theta_i \partial^\theta_j = 
    - \partial^\theta_j \partial^\theta_i
\end{equation}
for $1 \leq i, j \leq n$.
 Given $f = f(x_1, \dots, x_n, \theta_1, \dots, \theta_n) \in \Omega_n$, we therefore have a well-defined differential operator on $\Omega_n$:
 \begin{equation}
 \partial f \coloneqq f(\partial_1, \dots, \partial_n, \partial_1^\theta, \dots, \partial_n^{\theta}).
 \end{equation}
 which in turn gives rise to an action
 \begin{equation}
(-) \odot (-): \Omega_n \times \Omega_n \rightarrow \Omega_n
 \end{equation}
 of superspace on itself by $f \odot g \coloneqq \partial f(g).$

  Let $I \subseteq \Omega_n$ be a bihomogeneous ideal. The {\em Macaulay inverse system} $I^{\perp}$ is the vector subspace of $\Omega_n$ given by
  \begin{equation}
  I^{\perp} \coloneqq  \{ g \in \Omega_n \,:\,  f \odot g = 0 \text{ for all $f \in I$ } \}
  \end{equation}
  It is not hard to see that $I^{\perp}$ is a bigraded subspace of $\Omega_n$.

  Given $f \in \Omega_n$, define $\bar{f} \in \Omega_n$ by reversing the order of the $\theta$-monomials in $f$ and taking the complex conjugate of the coefficients in $f$.
  The map $\overline{\boldsymbol{\cdot}}: \Omega_n \rightarrow \Omega_n$ is an $\mathbb{R}$-linear involution. 
  We have an inner product on $\Omega_n$ given by
 \begin{equation}
 \langle f, g \rangle \coloneqq \text{constant term of } f \odot \overline{g}.
 \end{equation}
  The bidegree direct sum decomposition $\Omega_n = \bigoplus_{i,j \geq 0} \left( \Omega_n \right)_{i,j}$ is orthogonal with respect to this pairing.
  If $I \subseteq \Omega_n$ is a bihomogeneous ideal,
  we have 
  \begin{equation}
      \Omega_n = I \oplus \overline{I^\perp}
  \end{equation}
  and
  the composite map
  \begin{equation}
  I^{\perp} \hookrightarrow \Omega_n \xrightarrow{ \, \, \overline{\boldsymbol{\cdot}} \, \, } \Omega_n \twoheadrightarrow \Omega_n/I
  \end{equation}
  is a bijection. 

  Let $V = \bigoplus_{i \geq 0} V_i$ be a 
  graded complex vector space with each piece
  $V_i$ finite-dimensional.
  If $q$ is a variable, recall that 
  the {\em Hilbert series} of $V$ is the formal 
  power series
  \begin{equation}
      \Hilb(V;q) \coloneqq \sum_{i \geq 0}
      \dim_\CC(V_i) \cdot q^i.
  \end{equation}
  More generally,
  let $V = \bigoplus_{i,j \geq 0} V_{i,j}$ be a 
  bigraded complex vector space with each piece $V_{i,j}$
  finite-dimensional.
  The {\em bigraded Hilbert series} of $V$ is 
  \begin{equation}
      \Hilb(V;q,t) \coloneqq 
      \sum_{i,j \geq 0} \dim_\CC(V_{i,j}) \cdot q^i t^j
  \end{equation}
  in the variables $q,t$. In this paper,
  the space $V$ will always be a bihomogeneous
  subspace or quotient of superspace $\Omega_n$,
  the variable $q$ will track bosonic degree, and the 
  variable $t$ will track fermionic degree.

  If $I \subseteq \Omega_n$ is a bihomogeneous ideal,
  we have the bigraded direct 
  sum $\Omega_n = I \oplus \overline{I^\perp}$. 
  The bigraded Hilbert series of the quotient ring
  $\Omega_n/I$ therefore coincides with that of the inverse
  system $I^\perp$:
  \begin{align}
      \Hilb(\Omega_n/I;q,t) = \sum_{i,j\geq 0}\dim_{\CC}\left(\Omega_n/I\right)_{i,j} 
      \cdot q^it^j = 
      \sum_{i,j \geq 0} \dim_\CC(I^\perp) \cdot q^i t^j
      = \Hilb(I^\perp;q,t).
  \end{align}

\subsection{Multiarrangements}
\label{subsec:multi}
Fix a positive integer $n$.
A {\em linear hyperplane} $H$ is a codimension one subspace of $\CC^n$. 
Any linear hyperplane $H$ is the zero set of a 
homogeneous linear form $\alpha_H=a_1x_1+\cdots a_nx_n$ for $(a_1,\dots,a_n)\in \CC^n\setminus \{0\}$.
We refer to $\alpha_H$
as the {\em normal vector} of $H$; it is 
unique
up to a nonzero scalar.

An {\em affine hyperplane} $H$
is an affine translate of a linear hyperplane;
the normal vector of an affine hyperplane is 
that of its linear translate.
A {\em multiarrangement} $\AAA$ of hyperplanes is a collection $\{H_1,\dots,H_m\}$ of hyperplanes where we allow for repeats. 
For a given hyperplane $H$ in $\AAA$, we refer to the number of times it appears in the collection $\AAA$ as its \emph{multiplicity}.
An arrangement is \emph{simple} if the multiplicity of any hyperplane in it is equal to $1$.
Henceforth, we employ the term `arrangement' for multiarrangements.

Let $\AAA$ be a multiarrangement and let $H \in \AAA$
be a hyperplane.
The {\em deletion} $\AAA - H$ is the multiarrangement
obtained from $\AAA$ by removing one copy of $H$.
The {\em restriction}
$\AAA \mid H$ is the arrangement
\begin{equation}
    \AAA \mid H \coloneqq \{ H' \cap H \,:\, H' \in \AAA - \{H\},
    \, \, H' \cap H \neq \varnothing \}.
\end{equation}

We need some matroidal terminology for hyperplane arrangements. 
By recording the normal vectors for $\AAA$ as the columns of a matrix, we obtain the realizable matroid $\mathfrak{M}_{\AAA}$ corresponding to $\AAA$. Under this association, matroidal notions are inherited by $\AAA$.
The \emph{rank} of $\AAA$ is the rank of the matrix of normal vectors, or equivalently the dimension of the space spanned by the normal vectors of the hyperplanes composing $\AAA$.
Naturally, any subset of the hyperplanes can now be endowed with a rank.
A {\em basis}  of $\AAA$ is a subset $B \subset \AAA$ such that $\{ \alpha_H \,:\, H \in B \}$ forms a basis of the space spanned by
$\{ \alpha_H \,:\, H \in \AAA \}$.

If $\AAA$ is a linear multiarrangement, then a hyperplane $H \in \AAA$ is a 
 {\em coloop} if a linear form $\alpha_H$ with kernel $H$ does not lie in the span of the linear forms $\{ \alpha_{H'} \,:\, H' \in \AAA, \, \, H' \neq H \}$
of the other hyperplanes in $\AAA$.
A ``hyperplane"  $H \in \AAA$ is a {\em loop} if its linear form is $\alpha_H = 0$. Loops can arise from restrictions of loop-free multiarrangements.

    Given a matroid $\MMM$ on a ground set $\groundset$ with rank function $\mathsf{rk} :2^{\groundset}\to \mathbb{Z}_{\geq 0}$,  its \emph{Tutte polynomial} $T_{\MMM}(x,y)\in \mathbb{N}[x,y]$ is 
    \begin{align}\label{eq:defn_tutte}
        T_{\MMM}(x,y)=\sum_{A\subseteq \groundset} (x-1)^{r-\rank{A}} (y-1)^{|A|-\rank{A}},
    \end{align}
    where $r\coloneqq \rank{\MMM}$.
    Tutte \cite{Tut54} famously showed that this implies a combinatorial expansion in terms of \emph{internal} and \emph{external activities} as we range over the set $\mathcal{B}$ of all  bases of $\mathfrak{M}$:
    \begin{align}\label{eq:defn_tutte_activity}
    T_{\mathfrak{M}}(x,y)=\sum_{\text{bases } \mathcal{B}} x^{\intact(B)}y^{\extact(B)}.    
    \end{align}
    We postpone the formal definition of the notion of activity (and passivity) until Section \ref{sec:basis}.
    If the matroid $\MMM$ is obtained from an arrangement $\AAA$ as outlined earlier, we may talk about the Tutte polynomial $T_{\AAA}$ without any ambiguity.

\section{Superpower Ideals and Spanning 
Sets}\label{sec:spanning}

\subsection{Superpower ideals}
\label{subsec:superpower}

Given $(\ell_1,\dots,\ell_n)\in \CC^n\setminus {0}$, consider the line $L = \CC \cdot (\ell_1, \dots, \ell_n)$ in $\CC^n$ and let $\lambda_L$ be the linear form 
\begin{equation}
    \lambda_L \coloneqq \ell_1 x_1 + \cdots + \ell_n x_n \in \CC[\xx_n].
\end{equation}
The linear form $\lambda_L$ is defined up to a nonzero
scalar.

Let $\AAA=\{H_1,\dots,H_m\}$  be a multiarrangement of linear hyperplanes in $\CC^{n}$ and let $k \geq -1$
be an integer. 
Ardila--Postnikov \cite[Section 3.1]{AP10} defined
the ideal
$J_{\AAA,k} \subseteq \CC[\xx_n]$ by
\begin{align}\label{eq:power_ideals}
    J_{\AAA,k}\coloneqq \left (\lambda_L^{\rho_{\AAA}(L)+k}\suchthat L \subseteq \CC^n \text{ a line} \right)
\end{align}
where 
\begin{equation}
\rho_\AAA(L) \coloneqq \# \{ 1 \leq i \leq m \suchthat  L \not\subseteq H_i \}.
\end{equation}
The cases $k\in \{-1,0,1\}$ are of particular interest as the corresponding  quotients $\CC[\xx_n]/J_{\AAA,k}$ respectively give the \emph{internal}, \emph{central} and \emph{external zonotopal algebras}. 
One reason for this interest stems from the fact that the (singly-graded) Hilbert series of these algebras are obtained as  univariate specializations of the Tutte polynomial of $\mathfrak{M}_{\AAA}$.
More precisely if $r$ is the rank of $\AAA$ we have \cite{AP10, HR11}
\begin{align}\label{eq:hilb_classical_central}
    \Hilb(  \CC[\xx_n]/J_{\AAA,0};q) &= q^{m-r}  
    T_{\AAA}(1,q^{-1}),
    \\ \label{eq:hilb_classical_external}
    \Hilb(  \CC[\xx_n]/J_{\AAA,1};q) &= q^{m-r} T_{\AAA}\left(1+q,q^{-1}\right).
\end{align}
The following superspace ideals are our 
object of study.

\begin{definition}
    Let $\AAA$ be a multiarrangement in $\CC^n$.
    We let $I_\AAA \subseteq \Omega_n$ be the 
    differential closure of the ideal $J_{\AAA,1}$.
    In terms of generators, we have 
    \begin{equation}
    \label{eq:main_superspace_ideal_of_interest}
    I_\AAA \coloneqq  \left( \lambda_L^{\rho_\AAA(L) + 1}, d \lambda_L^{\rho_\AAA(L) + 1} \,:\, L \subseteq \CC^n \text{ a line}  \right) \subseteq \Omega_n.
    \end{equation}
    The ideal $I_\AAA$ is bihomogeneous.
    We write
    \begin{equation}
        \external_{\AAA}\coloneqq \Omega_n/I_{\AAA}
    \end{equation}
    for the associated bigraded quotient ring and
    \begin{equation}
        I_\AAA^\perp \subseteq \Omega_n
    \end{equation}
    for the bigraded subspace of $\Omega_n$ given by 
    the
    inverse system of $I_\AAA$.
\end{definition}

We will be interested in the bigraded Hilbert series
$\Hilb(\external_\AAA;q,t)$.
By setting the $\theta$-variables equal to zero,
we see from 
Equation~\eqref{eq:hilb_classical_external} that
\[
    \Hilb(\external_\AAA;q,0) =
    \Hilb(I_\AAA^\perp;q,0) = 
    \Hilb(\CC[\xx_n]/J_{\AAA,1};q) = 
    q^{m-r} T_{\AAA}\left(1+q,q^{-1}\right)   
\]
Less obviously (see Corollary~\ref{cor:hilb_cases}),
the top $t$-degree of 
$\Hilb(\external_\AAA;q,t)$ is a polynomial in $q$
coinciding with 
$\Hilb(\CC[\xx_n]/J_{\AAA,0};q)$.
Various features of $\external_\AAA$
interpolate between the external and central zonotopal
algebras.

\begin{example}\label{ex:first}
    Consider the hyperplane arrangement $\AAA$ in $\CC^2$ determined by the hyperplane $x_1-x_2=0$. For $L=\CC\cdot (a,b)$ with $a\neq b$, we know that $\rho_{\AAA}(L)=1$. If $a=b \neq 0$, on the other hand, then $\rho_{\AAA}(L)=0$. In turn, this means that the ideal $I_{\AAA}$ is generated by $(ax_1+bx_2)^2$ where $a\neq b$ are not both $0$, $x_1+x_2$, as well as their differentials. Equivalently we have
    \[
        I_{\AAA}=(x_1^2,x_2^2,x_1+x_2,\theta_1x_1,\theta_2x_2,\theta_1+\theta_2).
    \]
    It can be checked that a monomial basis for $\external_{\AAA}$ is given by $\{1,x_1,\theta_1\}$, and so $\Hilb(\external_{\AAA};q,t)=1+q+t$.
\end{example}

\begin{example}
\label{ex:second}
    Let us consider a slightly more involved example where $\AAA$ is determined by $x_1=0$ and $x_1-x_2=0$. One can show that
    \[
        I_{\AAA}=(x_1^3,x_2^2,(x_1+x_2)^2,\theta_1x_1^2,\theta_2x_2,(\theta_1+\theta_2)(x_1+x_2)).
    \]
    This time we get a monomial basis given by elements of 
    \[
    \{1,x_1,x_2,\theta_1,\theta_2,x_1^2,\theta_1x_1,\theta_1x_2,\theta_1\theta_2\},
    \]
    and therefore
    \[
    \Hilb(\external_{\AAA}) = 1+2q+2t+2qt+q^2+t^2=(1+q+t)^2.
    \]
    We invite the reader to check that a basis for $I_{\AAA}^{\perp}$ is given by 
    \[
    \{1, x_1, x_1-x_2, \theta_1, \theta_1-\theta_2, x_1(\theta_1-\theta_2), (x_1-x_2)\theta_1, x_1(x_1-x_2) ,\theta_1(\theta_1-\theta_2)\}.
    \]
    The careful reader may have noted the similarity between the expressions that appear in this list above and the defining hyperplanes of the arrangement $\AAA$.
\end{example}

We will approach the Hilbert series of $\external_{\AAA}$ by way of understanding $I_{\AAA}^{\perp}$.
  To this end, we first find a spanning set for $I_{\AAA}^{\perp}$.

\subsection{Derivative lemmata}
The ideal $I_\AAA$ is generated by certain powers
$(\lambda_L)^p$ of 
linear forms  corresponding to lines,
and their total derivatives 
$(\lambda_L)^{p-1} \cdot d \lambda_L$.
In order to understand the inverse system
$I_\AAA^\perp$, it is therefore important to understand
 the operators $\lambda_L \odot (-)$ and
$d \lambda_L \odot (-)$. 
The following simple lemma is the basis of our analysis.

\begin{lemma}
\label{linear-vanishing}
Let $H\subseteq \CC^n$ be a linear hyperplane determined by a homogeneous linear form 
$\alpha_H \coloneqq a_1 x_1+\cdots+ a_n x_n$.  
Consider a line $L = \CC \cdot (\ell_1, \dots, \ell_n)$ with associated linear form
$\lambda_L = \ell_1 x_1 + \cdots + \ell_n x_n$.
\begin{enumerate}[label=\textup{(}\arabic*\textup{)}]
\item \label{it1}  We have $(d \lambda_L) \odot \alpha_H = \lambda_L \odot (d \alpha_H) = 0$.
\item \label{it2} We have $\lambda_L \odot \alpha_H = (d \lambda_L) \odot (d \alpha_H) \in \CC$, and this complex number is zero if and only if $L \subseteq H$.
\end{enumerate}
\end{lemma}
\begin{proof}
\ref{it1} is true for degree reasons. 
For \ref{it2}, and for the second, note that $L \subseteq H$ if and only if $a_1 \ell_1 + \cdots + a_n \ell_n = 0$ and  $\lambda_L \odot \alpha_H = (d \lambda_L) \odot (d \alpha_H) = a_1 \ell_1 + \cdots  + a_n \ell_n$. 
\end{proof}

Using Lemma~\ref{linear-vanishing} we can
record some superspace 
elements which lie in $I_\AAA^{\perp}$. The following 
proof makes use of the product rule for superspace
differentiation. If $\lambda \in \CC[\xx_n]_1$
is a homogeneous linear polynomial and $f, g$
are bihomogeneous superspace elements, we have
\begin{equation}
    \label{eq:product-rule}
    \begin{cases}
        \lambda \odot (f \cdot g) = 
        (\lambda \odot f) \cdot g + 
        f \cdot (\lambda \odot g) \\
        (d \lambda) \odot (f \cdot g) = 
        ((d \lambda) \odot f) \cdot g \pm
        f \cdot ((d \lambda) \odot g)
    \end{cases}
\end{equation}
where the $\pm$ in the second branch is the parity
of the fermionic degree of $f$.
 
 \begin{lemma}
 \label{perp-membership}
 Let $\alpha_1, \dots, \alpha_m \in \CC[\xx_n]$ be homogeneous linear forms defining the linear hyperplanes $H_1, \dots, H_m$.
Suppose we have a disjoint union decomposition $[m] = I \sqcup J$. For any subset $S \subseteq J$, the superspace element
\[
\prod_{i \in I} d \alpha_i \times \prod_{s \in S} \alpha_s
\]
lies in $I_\AAA^{\perp}$.
 \end{lemma} 
 \begin{proof}
 Consider a line $L = \CC \cdot (\ell_1, \dots, \ell_n)$ in $\CC^n$.
 Let $\rho\coloneqq \rho_\AAA(L)$ throughout this proof and let $f_{I,S}\coloneqq \prod_{i \in I} d \alpha_i \times \prod_{s \in S} \alpha_s$ denote the superspace element in question.
 We need to show that $\lambda_L^{\rho+1}$ and $d \lambda_L \times \lambda_L^{\rho}$ both annihilate $f_{I,S}$ 
 under the $\odot$-action.
 
 Applying the superspace analogue 
 \eqref{eq:product-rule}
 of the product rule and Lemma~\ref{linear-vanishing}\ref{it1}, 
 we calculate 
 \begin{align}
 \lambda_L^{\rho + 1} \odot f_{I,S} &= 
 \prod_{i \in I} d \alpha_i  \times 
 \sum_{\substack{T \subseteq S \\ |T| = \rho + 1}} \left( \prod_{t \in T} (\lambda_L \odot \alpha_t) \times \prod_{s \in S - T} \alpha_s \right), \label{eq1}
 \\
  \left( d\lambda_L \times \lambda_L^{\rho} \right) \odot f_{I,S}  &=  
  \left(
  \sum_{i_0 \in I}  \pm (d \lambda_L \odot d \alpha_{i_0})
  \prod_{i \in I - i_0} d \alpha_i \right) \times
  \left(
  \sum_{\substack{T \subseteq S \\ |T| = \rho}} 
   \left( \prod_{t \in T} (\lambda_L \odot \alpha_t) \times \prod_{s \in S - T} \alpha_s \right) \right). \label{eq2}
 \end{align}
 where the sign $\pm$ depends on $i_0$ and arises from permutation of the fermionic factors.

 On the right-hand side of Equation~\eqref{eq1}, since $|T|=\rho+1$, there is at least one hyperplane $H_t$ for $t\in T$ that contains $L$.
 Lemma~\ref{linear-vanishing}\ref{it2} then guarantees that at least one of the factors $(\lambda_L \odot \alpha_t)$ will vanish for every summand $T$ and we have $\lambda_L^{\rho+1}\odot f_{I,S}=0$ as desired.  
 
 We apply a similar analysis to the right-hand side of Equation~\eqref{eq2}.  
 For a fixed subset $T \subseteq S$, if none of the factors $\lambda_L \odot \alpha_t$ vanish, we know by Lemma~\ref{linear-vanishing}\ref{it2} that the hyperplanes $H_t$ for $t\in T$ are precisely the $\rho$ hyperplanes in our arrangement that do not contain the line $L$. But then this implies that for every $i_0 \in I$
 the line $L$ is contained in
 $H_{i_0}$ so that 
 $(d \lambda_L \odot d \alpha_{i_0}) = 0$. Thus we see that $(d\lambda_L\times\lambda_L^{\rho})\odot f_{I,S}=0$ as well, thereby concluding the proof.
 \end{proof}

 \subsection{The case of a coordinate arrangement}
 Going back to the second half of Example~\ref{ex:second}, note that the basis of $I_{\AAA}^{\perp}$ contains all elements of the form in Lemma~\ref{perp-membership}.
 In fact, the elements appearing in 
 Lemma~\ref{perp-membership} will always 
 span the vector space $I_\AAA^{\perp}$.
 The argument showing this will be inductive. To get 
 this induction off  the ground, 
 we consider the case where $\AAA$ is a multiarrangement consisting entirely of coordinate hyperplanes.
 
 \begin{lemma}
 \label{coordinate-case}
 Let $(a_1, \dots, a_p)$ be a list of $p$ positive integers satisfying $a_1 + \cdots + a_p = m$ where $p \leq n$ and let $\AAA$ 
 be the multiarrangement in $\CC^n$ consisting of $a_i$ copies of the coordinate hyperplane $x_i = 0$ for each $1 \leq  i \leq p$. Then
 $$
 I_\AAA = ( x_1^{a_1 + 1}, \dots, x_n^{a_n + 1}, x_1^{a_1} \theta_1, \dots, x_n^{a_n} \theta_n) 
 $$
 where we interpret $a_i = 0$ for $i > p$.
 Consequently, the Macaulay inverse $I_\AAA^{\perp}$ has basis
 consisting of monomials $x_1^{i_1} \cdots x_n^{i_n}  \theta_S$ where $S \subseteq [p]$ and additionally
 \begin{itemize}
 \item $i_j \leq a_j$, and
 \item $i_s < a_j$ for all $s \in S$.
 \end{itemize}
 \end{lemma}
 
 \begin{proof}
 We start by reducing to the case $p = n$.  If $p < n$, for any $p + 1 \leq i  \leq n$ the line $L_i = \CC \cdot e_i$ is contained in each hyperplane of $\AAA$. Here $e_i$ denotes the standard basis vector of $\CC^n$ with $1$ in the $i$th coordinates and zeros everywhere else.
 Thus for each such $L_i$ we have $\rho_{\AAA}(L_i)=0$, allowing us to conclude that 
 \begin{align}
 \label{eq:tail_in_ideal}
     x_i, \theta_i \in I_\AAA \text{ for } p+1\leq i\leq n.
 \end{align}

 Now consider any line $L = \CC \cdot (\ell_1, \dots, \ell_n)$ and, like before, set $\rho=\rho_{\AAA}(L)$. 
 From~\eqref{eq:tail_in_ideal} we have the equalities modulo $I_\AAA$:
 \begin{align}
 \label{coordinate-case-one}
 \lambda_L^{\rho + 1} &= ( \ell_1 x_1 + \cdots + \ell_p x_p)^{\rho + 1}
 \\\label{coordinate-case-two}
  d\lambda_L\times \lambda_L^{\rho}
  &= ( \ell_1 x_1 + \cdots + \ell_p x_p)^{\rho} \times (\ell_1 \theta_1 + \cdots + \ell_p \theta_p).
 \end{align}
 Furthermore, if $L' = \CC \cdot (\ell_1, \dots, \ell_p, 0, \dots, 0)$ then we have $ \rho_\AAA(L')$ also equals $\rho$.\footnote{If $\ell_1 = \cdots = \ell_p = 0$, then
 both \eqref{coordinate-case-one} and \eqref{coordinate-case-two} lie in $I_\AAA$.}
 We may therefore assume that $p = n$ going forward.
 
 The elements in the proposed generating set of $I_\AAA$ are precisely the generators of $I_\AAA$ coming from the lines $\CC\cdot e_i$ given by the $n$ coordinate axes 
 of $\CC^n$, so the containment 
 $$ ( x_1^{a_1 + 1}, \dots, x_n^{a_n + 1}, x_1^{a_1} \theta_1, \dots, x_n^{a_n} \theta_n)\subseteq I_{\AAA}$$
 is clear.
 
 For the reverse containment, let 
 $L = \CC \cdot (\ell_1, \dots, \ell_n)$ be an arbitrary line in $\CC^n$ and let $S = \{ 1 \leq i \leq n \,:\, \ell_i \neq 0 \}$.
 Then $\rho = \rho_{\AAA}(L) = \sum_{i \in S} a_i$
 and we have the generator
 $\lambda_L^{\rho + 1}\in I_\AAA$.  
 Expanding the latter into monomials, we see that
 \begin{equation}
 \lambda_L^{\rho + 1} = 
  \lambda_L^{1 + \sum_{i \in S} a_i}=\left( \sum_{i\in S}\ell_ix_i\right)^{1 + \sum_{i \in S} a_i} \in ( x_i^{a_i + 1} \,:\, i \in S ).
 \end{equation}
 Indeed, any monomial $\prod_{i\in S}x_i^{m_i}$ appearing with nonzero coefficient has degree $1 + \sum_{i \in S} a_i$, which means that at least one $m_i$ satisfies $m_i\geq 1+a_i$.

 Similarly, expanding the product
  $\lambda_L^{\rho} \times d\lambda_L$ into monomials 
  leads to 
  \begin{align*}
    \left( \sum_{i \in S} \ell_i x_i \right)^{\sum_{i \in S} a_i} \times \left( \sum_{i \in S} \ell_i \theta_i \right)
    \in
    \left( x_i^{a_i} \theta_i \,:\, i \in S  \right) + \left( x_i^{a_i + 1} \,:\, i \in S \right)  
  \end{align*} 
  where we used that $a_i > 0$ for all $i$. Indeed, this time monomials with nonzero coefficient are of the form $\theta_j\prod_{i\in S}x_i^{m_i}$ for some $j\in S$, with bosonic degree $\sum_{i\in S}a_i$. If no $m_i$ satisfies $m_i\geq 1+a_i$, then we necessarily must be in the situation $m_i=a_i$ for all $i\in S$. But then any such monomial is in the ideal generated by $x_j^{a_j}\theta_j$.
  This completes the proof of the reverse containment.
  
 Given our presentation of $I_\AAA$, it is not hard to see that $I^{\perp}_\AAA$ has the required basis.
 We leave the details to the reader.
 \end{proof}

\subsection{A spanning set via an exact sequence}
\label{subsec:spanning}
 
 Let us return to the case of an arbitrary arrangement $\AAA = \{ H_1, H_2, \dots, H_m \}$ of $m$ linear hyperplanes in $\CC^n$.
 Changing coordinates if necessary, \emph{we may assume that $H_1$ is given by $x_1 = 0$ and we will work under this assumption for the remainder for this section.}
 
 Let $\AAA - H_1 = \{ H_2, \dots, H_m \}$ and $\AAA \mid H_1 = \{ H_j \cap H_1 \,:\, j \geq 2 \}$ be the deletion and restriction of $\AAA$ with respect to $H_1$.
We propose to define an exact sequence 
 \begin{equation}
 \label{exact-sequence}
 0 \rightarrow I^{\perp}_{\AAA - H_1}  \xrightarrow{ \, \, \varphi \, \, } I^{\perp}_\AAA \xrightarrow{ \, \, \psi \, \,} I^{\perp}_{\AAA \mid H_1} \oplus I^{\perp}_{\AAA \mid H_1} \rightarrow 0
 \end{equation}
 where
 \begin{itemize}
 \item the map $\varphi: I^{\perp}_{\AAA - H_1}  \rightarrow I^{\perp}_\AAA$ is given by $\varphi(f) \coloneqq x_1 \cdot f$, and
 \item the map $\psi: I^{\perp}_\AAA \rightarrow I^{\perp}_{\AAA \mid H_1} \oplus I^{\perp}_{\AAA \mid H_1}$ is given by 
 \begin{align}\label{eq:non_standard}
 \psi(f) = ( f \mid_{x_1, \theta_1 \, = \, 0}, \theta_1 \odot f \mid_{x_1 \, = \, 0} ).
 \end{align}
 \end{itemize}
 In the second coordinate of $\psi$, 
 the notation
 $\theta_1 \odot f \mid_{x_1 \, = \, 0}$
 is justified because differentiation
 $\theta_1 \odot (-)$ with respect
 to $\theta_1$ commutes with 
 evaluation $(-) \mid_{x_1 \, =\,  0}$
 at $x_1 = 0$.
 
At this stage, it is not even obvious that the 
maps $\varphi, \psi$ involved in
\eqref{exact-sequence} are well-defined; we will prove this in Lemma \ref{lem:well_defined}.
Granting the definedness and exactness 
of \eqref{exact-sequence} for the moment, observe that 
$\varphi$ increases bosonic degree by 1 whereas $\psi$
preserves degrees in the first component while reducing
fermionic degree by 1 in the second component.
Said differently, the sequence \eqref{exact-sequence}
gives rise to a homogeneous exact sequence
of bigraded $\CC$-vector spaces
\begin{equation}
    0 \longrightarrow 
    I^\perp_{\AAA - H_1}(-1,0) \longrightarrow
    I^\perp_\AAA \longrightarrow 
    I^\perp_{\AAA \mid H_1}(0,-1) \oplus I^\perp_{\AAA \mid H_1} \longrightarrow 0
\end{equation}
where $I^\perp_{\AAA - H_1}(-1,0)$ is 
$I^\perp_{\AAA - H_1}$ with bosonic degree increased by 1
and $I^\perp_{\AAA \mid H_1}(0,-1)$ is 
$I^\perp_{\AAA \mid H_1}$ with fermionic degree
increased by 1. In terms of bigraded Hilbert series, this
implies
\begin{equation}
\label{eq:technical-hilbert-series-recursion}
    \Hilb(\external_\AAA;q,t) = q \cdot 
    \Hilb(\external_{\AAA - H_1};q,t)
    + (1+t) \cdot \Hilb(\external_{\AAA \mid H_1};q,t).
\end{equation}
We use
Equation~\eqref{eq:technical-hilbert-series-recursion}
to relate the Hilbert series of 
$\external_\AAA$ to the Tutte polynomial of $\AAA$.

We remark that exact sequences analogous to
\eqref{exact-sequence} have appeared in the purely
bosonic context before.
For example, the sequence \eqref{exact-sequence} should 
be compared with the sequence
appearing in the proof of \cite[Prop. 4.4]{AP18}.
The key difference in the superspace context is the 
presence of 
$I_{\AAA \mid H_1}^\perp(0,-1) \oplus I_{\AAA \mid H_1}^\perp$ 
instead of just $I_{\AAA \mid H_1}^\perp$ on the right.
Before establishing the exactness of 
\eqref{exact-sequence}, let us consider an example.
\begin{example}
    Let $\AAA$ be the multiarrangement in $\CC^2$ determined by two copies of the hyperplane $H$ given by $x_1=0$. 
    The reader may check that a basis for $I_{\AAA}^{\perp}$ is given by 
    \[
    \{1,x_1,x_1^2,\theta_1,\theta_1x_1\}.
    \]
    The arrangement $\AAA-H$ is determined by a single copy of $H$ and in this case we have a basis for $I_{\AAA-H_1}^{\perp}$ given by $\{1,x_1,\theta_1\}$. Multiplying each element by $x_1$ produces the subset $\{x_1,x_1^2,x_1\theta_1\}$ of the basis for $I_{\AAA}^{\perp}$.
    Restricting $\AAA$ to $H$ produces a `degenerate' hyperplane which may be deleted. Thus we have that $I_{\AAA\mid H}^{\perp}=\CC\{1\}$. Finally, observe that the only basis element that survives the specialization $x_1=0$ and $\theta_1=0$ is $1$, and the unique element that that survives the $\odot$ action of $\theta_1$ and then the specialization $x_1=0$ is $\theta_1$ itself. Thus we see in this simple scenario how the basis elements of $I_{\AAA}^{\perp}$, $I_{\AAA-H}^{\perp}$, and $I_{\AAA\mid H}^{\perp}$ interact with the maps $\varphi$ and $\psi$. 
\end{example}
A more thorough and intricate analysis along these lines shall be performed in Section~\ref{sec:basis}.
For the moment, we undertake a finer investigation of the maps $\varphi$ and $\psi$.

 \begin{lemma}
 \label{lem:well_defined}
     The maps $\varphi$ and $\psi$ are well-defined.
 \end{lemma}
 \begin{proof}
We begin by arguing that $\varphi$ is well-defined.  Thus we would like to show that for $f\in I_{\AAA-H_1}^{\perp}$ we must have $x_1\cdot f \in I_{\AAA}^{\perp}$.
Equivalently, we want to show that $x_1\cdot f$ is annihilated by $\lambda_{L}^{\rho_{\AAA}(L)+1}$ and $d\lambda_L\times \lambda_{L}^{\rho_{\AAA}(L)}$ under the $\odot$ action for any line $L\subseteq \CC^n$.

The product rule 
\eqref{eq:product-rule}
implies that for all $p>0$ and all elements $f \in \Omega_n$, we have
\begin{align}
    \label{line-leibniz1}
    \lambda_L^p \odot (x_1 \cdot f) &= (\lambda_L \odot x_1) \cdot (\lambda_L^{p-1} \odot f) + x_1 \cdot ( \lambda_L^p \odot f), \\
 d \lambda_L \odot (x_1\cdot  f) &= (d \lambda_L \odot x_1) \cdot f + x_1 \cdot (d \lambda_L \odot f)= x_1 \cdot (d \lambda_L \odot f)\label{line-leibniz2}.
\end{align}
For the second equality in~\eqref{line-leibniz2} we 
used Lemma~\ref{linear-vanishing}\ref{it1}.

Note that for any line $L \subseteq \CC^n$, we have
\begin{equation}\label{eq:rho_upon_deletion}
\rho_\AAA(L) = \begin{cases}
\rho_{\AAA - H_1}(L) & L \subseteq H_1, \\
\rho_{\AAA - H_1}(L)  + 1 & L \not\subseteq H_1.
\end{cases}
\end{equation}
We take cases based on this equality.

 Suppose $L\subseteq H_1$.
 Let us take $p=\rho_{\AAA}(L)+1$ in~\eqref{line-leibniz1}. 
 By Lemma~\ref{linear-vanishing}\ref{it2} we have $\lambda_L \odot x_1 = 0$.
 Additionally, since $f\in I_{\AAA-H_1}^{\perp}$, we know that $\lambda_L^{\rho_{A-H_1}(L) + 1}\odot f=0$. Thus both summands on the right-hand side of~\eqref{line-leibniz1} vanish, so that 
 $\lambda_L^{\rho_{\AAA}(L)+1}\odot (x_1\cdot f)=0$.
 As for $(\lambda_L^{\rho_{\AAA}(L)}\times d\lambda_L) \odot (x_1\cdot f)$, 
 Equation~\eqref{line-leibniz2} implies 
 \begin{equation}
 \label{eq:thing-we-want-to-kill}
     (\lambda_L^{\rho_{\AAA}(L)}\times d\lambda_L) \odot (x_1\cdot f) = 
     \lambda_L^{\rho_{\AAA}(L)} \odot (x_1 \cdot (d \lambda_L \odot f)).
 \end{equation}
 If $\rho_{\AAA}(L)=0=\rho_{\AAA-H_1}(L)$, 
 we have $d \lambda_L \in I_{\AAA - H_1}$
 and because $f \in I_{\AAA - H_1}^\perp$
 we know that $d\lambda_L \odot f=0$ and
 \eqref{eq:thing-we-want-to-kill} vanishes.
 If $\rho_{\AAA}(L)>0$ on the other hand, then~\eqref{line-leibniz1} with $p=\rho_{\AAA}(L)$ and $d\lambda_L\odot f$ in place of $f$ again results in the right-hand side
 of \eqref{eq:thing-we-want-to-kill}
 equaling $0$ since $d \lambda_L \cdot \lambda_L^p \in I_{\AAA - H_1}$.

 Now assume $L\not\subseteq H_1$.
 If $p=\rho_{\AAA}(L)+1=\rho_{\AAA-H_1}(L)+2$,
 we have $\lambda_L^{p-1} \in I_{\AAA - H_1}$
 and because $f \in I_{\AAA - H_1}^\perp$
 we have $\lambda_L^{p-1}\odot f=\lambda_L^p\odot f=0$.
 Equation~\eqref{line-leibniz1} therefore reads
 \begin{multline}
     \lambda_L^{\rho_\AAA(L) + 1} \odot (x_1 \cdot f) =
     (\lambda_L \odot x_1) \cdot 
     (\lambda_L^{\rho_{\AAA - H_1}(L) + 1} \odot f) +
     x_1 \cdot (\lambda_L^{\rho_{\AAA - H_1} + 2} \odot f)
    \\ = (\lambda_L \odot x_1) \cdot 0 + x_1 \cdot 0 = 0.
 \end{multline}
 As for $(\lambda_L^{\rho_{\AAA}(L)}\times d\lambda_L) \odot (x_1\cdot f)$, Equation~\eqref{line-leibniz2} implies
 \begin{equation}
     \label{eq:other-thing-we-want-to-kill}
     (\lambda_L^{\rho_{\AAA}(L)}\times d\lambda_L) \odot (x_1\cdot f) = 
     \lambda_L^{\rho_{\AAA}(L)} \odot (x_1 \cdot (d \lambda_L \odot f)).
 \end{equation}
 Setting $p=\rho_{\AAA}(L)=\rho_{\AAA-H_1}(L)+1$ in~\eqref{line-leibniz1} with $d\lambda_L\odot f$ in place of $f$ like before results in
 \begin{multline}
     \lambda_L^{\rho_\AAA(L)} \odot
     (x_1 \cdot (d \lambda_L \odot f)) =
     (\lambda_L \odot x_1) \cdot 
     ((\lambda_L^{\rho_{\AAA - H_1}(L)} \times d \lambda_L) 
     \odot f) + 
     x_1 \cdot 
     ((\lambda_L^{\rho_{\AAA - H_1}(L) + 1} \times d \lambda_L) 
     \odot f) \\
     = (\lambda_L \odot x_1) \cdot 0 + x_1 \cdot 0 = 0
 \end{multline}
 where we used 
 \begin{equation}
 (\lambda_L^{\rho_{A-H_1}(L)}\times d\lambda_L)\odot f=(\lambda_L^{\rho_{A-H_1}(L)+1}\times d\lambda_L)\odot f=0 
 \end{equation}
 as $f\in I_{\AAA-H_1}^{\perp}$.
We thus conclude that $x_1 \cdot f \in I^{\perp}_{\AAA}$ whenever $f \in I^{\perp}_{\AAA - H_1}$ and that the map $\varphi$ is well-defined.  

\smallskip
 
 Now we check that $\psi$ is well-defined.  
 Given $f \in \Omega_n$, we  have a unique
 finite expression of the form
 \begin{align}
 \label{eq:f_expansion}
 f = \sum_{i  \geq  0}  \left( x_1^i \cdot g_i + x_1^i \theta_1 \cdot h_i  \right)  
 \end{align}
 where $g_i, h_i \in \Omega_n$ are superspace elements which do not involve $x_1$ or $\theta_1$.  We would like to show that 
 if $f \in I_\AAA^\perp$ then $g_0, h_0 \in I_{\AAA \mid H_1}^\perp$.
 This will establish that the map
 $\psi: f \mapsto (g_0, h_0)$
  is a well-defined function
   $I_\AAA^\perp \rightarrow 
   I_{\AAA \mid H_1}^\perp \oplus I_{\AAA \mid H_1}^\perp$.

 Let $f \in I_\AAA^\perp$ and let 
 $L \subseteq H_1$ be a line contained in the hyperplane
 $H_1$.
 Then $\lambda_L^{\rho_\AAA(L) + 1}$ 
 and $d \lambda_L \times \lambda_L^{\rho_\AAA(L)}$
 are generators
 of $I_\AAA$ so that 
 \begin{equation}
 \label{eq:basic-perp-relationship}
     \lambda_L^{\rho_\AAA(L)+1}\odot f = 
     (d \lambda_L \times \lambda_L^{\rho_\AAA(L)}) \odot f = 0
 \end{equation}
 as $f \in I_\AAA^\perp$.
 Furthermore, since $L \subseteq H_1$
 neither $\lambda_L$ nor $d \lambda_L$ involve $x_1$ 
 or $\theta_1$ so that 
$\lambda_L \odot x_1 = (d \lambda_L) \odot \theta_1 = 0.$
 Combining Equations~\eqref{eq:f_expansion} and
 \eqref{eq:basic-perp-relationship} yields
\begin{align}\label{eq:vanishing_1}
 0 =   \lambda_L^{\rho_\AAA(L) + 1} \odot f = 
\sum_{i \geq 0}  \left( x_1^i \cdot (  \lambda_L^{\rho_\AAA(L) + 1} \odot g_i) + x_1^i \theta_1 \cdot  (\lambda_L^{\rho_\AAA(L) + 1} \odot h_i)  \right),
\end{align}
and similarly 
 \begin{align}\label{eq:vanishing_2}
 0 = 
 \sum_{i \geq 0}  \left( x_1^i \cdot  \left( \left( \lambda_L^{\rho_\AAA(L)}  \times  d \lambda_L \right) \odot g_i\right) + x_1^i \theta_1 \cdot \left( \left( \lambda_L^{\rho_\AAA(L)}  \times d \lambda_L \right)
 \odot h_i \right)  \right).
 \end{align}
 Extracting the coefficients of $x_1^0$ and $x_1^0 \theta_1$ in Equations~\eqref{eq:vanishing_1}  and~\eqref{eq:vanishing_2} gives
 \begin{equation*}
 \lambda_L^{\rho + 1}  \odot g_0 = \lambda_L^{\rho + 1} \odot h_0 = 
\left( \lambda_L^{\rho} \times d \lambda_L \right) \odot g_0 = 
\left( \lambda_L^{\rho} \times d \lambda_L \right) \odot h_0 = 0.
 \end{equation*}
 Since $L \subseteq H_1$ we have $\rho_\AAA(L) = \rho_{\AAA \mid H_1}(L)$ and we conclude that 
 $g_0, h_0 \in I^\perp_{\AAA \mid H_1}$.
 The function $\psi: f \mapsto (g_0, h_0)$ therefore
  a well-defined map
  $\psi: I_\AAA^\perp \rightarrow I_{\AAA \mid H_1}^\perp \oplus I_{\AAA \mid H_1}^{\perp}$.
 \end{proof}
 
Our next task is to show that the sequence of maps 
\eqref{exact-sequence}
is exact. Before doing so, we record a simple lemma
about derivatives.

 \begin{lemma}
     \label{derivative-implication-lemma}
     Let $L \subseteq \CC^n$ be a line, let $f \in \Omega_n$
     be a superspace element, and let 
     $\alpha = a_1 x_1 + \cdots + a_n x_n$
     be a nonzero linear form in $\CC[\xx_n].$
     Suppose that $\lambda_L^p \odot (\alpha \cdot f) = 0$
     for some $p \geq 0$. Then $\lambda_L^p \odot f = 0$.
 \end{lemma}

 \begin{proof}
     Changing coordinates if necessary, we may assume 
     $\lambda_L = x_1$.
     Given a superspace element $g \in \Omega_n$, we may write
     $g$ uniquely as a sum
      $g = \sum_{I \subseteq [n]} g_I \cdot \theta_I$
     for polynomials $g_I \in \CC[\xx_n]$.
     The assertion that $x_1^p \odot g = 0$ is equivalent 
     to the assertion that the $x_1$-degree of $g_I$ is 
     $< p$ for all $I \subseteq [n]$. 
     Since the $x_1$-degree of $\alpha \cdot g_I$ is $\geq$
     the $x_1$-degree of $g_I$ for all $I$, the lemma follows.
 \end{proof}

 With Lemma~\ref{derivative-implication-lemma} in hand, we 
 are ready to establish the exactness of 
 \eqref{exact-sequence} using an inductive argument.
 Intertwined with this induction, we establish an
 explicit spanning set for $I_\AAA^{\perp}$.
 
 \begin{lemma}
 \label{exact-spanning-lemma}
 The following statements hold:
 \begin{enumerate}[label=\textup{(}\arabic*\textup{)}]
 \item \label{it3}
 The set of polynomials in Lemma~\ref{perp-membership} is a spanning set of $I_\AAA^{\perp}$.
 \item \label{it4}
 The sequence \eqref{exact-sequence} is exact.
 \end{enumerate}
 \end{lemma}
 
 \begin{proof}
 Both claims shall follow from induction on $m\coloneqq |\AAA|$.
 Observe that~\ref{it3} follows from Lemma~\ref{coordinate-case} (and a change of variables) when $\AAA$ consists of $m \leq n$ hyperplanes corresponding to independent
 linear forms (even when these hyperplanes could have multiplicity $> 1$).
 
 The map $\varphi: I^{\perp}_{\AAA - H_1}  \rightarrow I^{\perp}_\AAA$ 
 of multiplication by $x_1$ is injective, 
 so the sequence \eqref{exact-sequence} 
 is certainly exact on the left. 
Furthermore, 
direct calculation shows that $\psi \circ \varphi = 0$ 
so that $$\mathrm{image}(\varphi) \subseteq \mathrm{ker}(\psi),$$
i.e. the sequence \eqref{exact-sequence} is a chain complex. 

Our next goal is to prove that \eqref{exact-sequence} is exact
in the middle.
Suppose $f \in I_\AAA^{\perp}$ satisfies $\psi(f) = 0$. 
We have $f = x_1 \cdot f'$ for some superspace 
element $f' \in \Omega_n$. We aim to show the membership
$f' \in I_{\AAA - H_1}^{\perp}$.
Since $f = x_1 \cdot f' \in I_\AAA^\perp$, we have
\begin{equation}
    \lambda_L^{\rho_\AAA(L) + 1} \odot (x_1 \cdot f') = 0 \quad \text{and}
    \quad
    (\lambda_L^{\rho_\AAA(L)} 
    \times d \lambda_L)  \odot (x_1 \cdot f') = 0
\end{equation}
for all lines $L$.
Applying the product rule 
\eqref{eq:product-rule}
to the second condition, 
this may be rewritten as 
\begin{equation}
\label{eq:vanishing-linear-derivatives-one}
    \lambda_L^{\rho_\AAA(L) + 1} \odot (x_1 \cdot f') = 0 \quad \text{and}
    \quad
    \lambda_L^{\rho_\AAA(L)} \odot (x_1 \cdot
    (d \lambda_L \odot f') = 0.
\end{equation}
If $L \subseteq H_1$, we have 
$\rho_\AAA(L) = \rho_{A - H_1}(L)$ so that 
\eqref{eq:vanishing-linear-derivatives-one} and
Lemma~\ref{derivative-implication-lemma} imply
\begin{equation}
\label{eq:vanishing-linear-derivatives-two}
    \lambda_L^{\rho_{\AAA - H_1}(L) + 1} \odot f' = 0 \quad \text{and}
    \quad
    \lambda_L^{\rho_{\AAA - H_1}(L)} \odot 
    (d \lambda_L \odot f') = 0.
\end{equation}
On the other hand, if $L \not\subseteq H_1$
we have $\rho_{\AAA - H_1}(L) = \rho_\AAA(L) - 1$ and  
$\lambda_L = a_1 x_1 + \cdots + a_n x_n$ for 
$a_1 \in \CC - \{0\}.$ 
We have
\begin{align}
    0 &= \lambda_L^{\rho_\AAA(L) + 1} \odot (x_1 \cdot f') \\
      &= (\lambda_L \odot x_1) \cdot 
         (\lambda_L^{\rho_\AAA(L)} \odot f') + 
         x_1 \cdot (\lambda_L^{\rho_\AAA(L) + 1} \odot f') \\
      &= a_1 \cdot (\lambda_L^{\rho_\AAA(L)} \odot f') + 
         x_1 \cdot (\lambda_L^{\rho_\AAA(L) + 1} \odot f') \\
     &= a_1 \cdot (\lambda_L^{\rho_\AAA(L)} \odot f') \\
     &= a_1 \cdot (\lambda_L^{\rho_{\AAA - H_1}(L) + 1} \odot f')
\end{align}
where the first equality follows from 
\eqref{eq:vanishing-linear-derivatives-one},
the second equality is the product rule, the third equality 
is the action of $\lambda_L \odot (-)$, the 
fourth equality follows from 
\eqref{eq:vanishing-linear-derivatives-one} and 
Lemma~\ref{derivative-implication-lemma}, and the final equality
uses $\rho_{\AAA - H_1}(L) = \rho_\AAA(L) - 1$.
Since $a_1 \neq 0$, we conclude that 
\begin{equation}
\label{eq:vanishing-linear-derivatives-three}
    \lambda_L^{\rho_{\AAA - H_1}(L) + 1} \odot f' = 0.
\end{equation}
The same reasoning with
$d\lambda_L \odot f'$ in the place of $f'$ shows that 
\begin{equation}
\label{eq:vanishing-linear-derivatives-four}
    (\lambda_L^{\rho_{\AAA - H_1}(L)} \times 
    d \lambda_L) \odot f' = 
    \lambda_L^{\rho_{\AAA - H_1}(L)} 
    \odot (d \lambda_L \odot f') = 0.
\end{equation}
Equations~\eqref{eq:vanishing-linear-derivatives-two},
\eqref{eq:vanishing-linear-derivatives-three}, and 
\eqref{eq:vanishing-linear-derivatives-four} imply that
 $f' \in I_{\AAA - H_1}^{\perp}$ 
 so that $f = x_1 \cdot f' = \varphi(f')$ lies
 in the image of $\varphi: I_{\AAA - H_1}^\perp \rightarrow I_\AAA^\perp$.
The sequence \eqref{exact-sequence} is therefore exact in the middle.

For exactness on the right, we need to show that 
$\psi: I^{\perp}_\AAA \rightarrow I^{\perp}_{\AAA \mid H_1} \oplus I^{\perp}_{\AAA \mid H_1}$  is surjective. To do this, we inductively assume that \ref{it3} holds for the restricted
arrangement $\AAA \mid H_1$.
This means that 
$I^{\perp}_{\AAA \mid H_1}$ is spanned by products of the form $\prod_{i \in I} d \widetilde{\alpha_i} \times \prod_{s \in S} \widetilde{\alpha_s}$ where $I, S$ are disjoint subsets of 
$\{2, \dots, m \}$ and $\widetilde{\alpha}$ is the restriction of a linear functional $\alpha \in (\CC^n)^*$ to $H_1$.
But Lemma~\ref{perp-membership} guarantees that both 
$\theta_1 \cdot \prod_{i \in I} d \alpha_i \times \prod_{s \in S} \alpha_s$ and
$\prod_{i \in I} d \alpha_i \times \prod_{s \in S} \alpha_s$ lie in $I_\AAA^{\perp}$, so that 
\begin{equation}
\label{first-psi-image}
\psi \left( \theta_1 \cdot \prod_{i \in I} d \alpha_i \times \prod_{s \in S} \alpha_s \right) = \left( 0, \prod_{i \in I} d \widetilde{\alpha_i} \times \prod_{s \in S} \widetilde{\alpha_s} \right)
\end{equation}
and 
\begin{equation}
\label{second-psi-image}
\psi \left( \prod_{i \in I} d \alpha_i \times \prod_{s \in S} \alpha_s \right) = \left(  \prod_{i \in I} d \widetilde{\alpha_i} \times \prod_{s \in S} \widetilde{\alpha_s}, 
\theta_1 \odot \left[  \prod_{i \in I} d \alpha_i \times \prod_{s \in S} \alpha_s    \right]_{x_1 = 0} \right)
\end{equation}
are in the image of $\psi$.  
Given the inductive structure of the spanning set of $I^{\perp}_{\AAA \mid H_1}$, we see from Equation~\eqref{first-psi-image} that 
\begin{equation}
\label{third-psi-image}
0 \oplus I^{\perp}_{\AAA \mid H_1} \subseteq \mathrm{image}(\psi).
\end{equation}
Equation~\eqref{second-psi-image}, the containment \eqref{third-psi-image}, and the inductive spanning set of $I^{\perp}_{\AAA \mid H_1}$ combine to show that 
\begin{equation}
I^{\perp}_{\AAA \mid H_1} \oplus I^{\perp}_{\AAA \mid H_1} = \mathrm{image}(\psi)
\end{equation}
so that $\psi$ is  surjective.
This completes the proof that the \eqref{exact-sequence} is exact.

It remains to show that the set of polynomials in
\ref{it3} spans $I_\AAA^\perp$.
We inductively assume that \ref{it3} holds for the 
$\AAA - H_1$ and $\AAA \mid H_1$.
The form of the exact sequence 
\eqref{exact-sequence}
means that the set of polynomials in 
\ref{it3} spans $I_\AAA^\perp$ as well.
 \end{proof}

 \section{The Hilbert series of $I_{\AAA}^{\perp}$}
 \label{sec:payoff}

 \subsection{Hilbert and Tutte}
Lemma~\ref{exact-spanning-lemma} translates to an equality at the level of Hilbert series. To use the so-called `recipe theorem' we need one additional lemma detailing how the Hilbert series are related in the presence of a coloop.
 
 \begin{lemma}
     \label{lem:coloop-coincidence}
     Suppose $H$ is a coloop in $\AAA$.
     We have $\Hilb(I_{\AAA-H}^\perp;q,t) = 
     \Hilb(I_{\AAA \mid H};q,t)$.
 \end{lemma}

 \begin{proof}
     Changing coordinates if necessary, we may assume that
     $H$ is the hyperplane $x_n = 0$, and that all other
     hyperplanes in $\AAA$ contain the line 
     $L_0 \coloneqq \CC \{ (0, \dots, 0, 1) \}$.
     Let $\pi: \CC^n \twoheadrightarrow \CC^n/H$
     be the canonical projection.
     If $L \subseteq \CC^n$ is a line with $L \neq L_0$ then
     $\pi(L)$ is a line in $\CC^n/H$ with
     \begin{equation}
     \label{eq:coloop-rho-relation-one}
         \rho_{\AAA - H}(L) = \rho_{\AAA \mid H}(\pi(L)).
     \end{equation}
     Furthermore, we have
     \begin{equation}
     \label{eq:coloop-rho-relation-two}
         \rho_{\AAA - H}(L_0) = 0
     \end{equation}
     so that $x_n, \theta_n \in I_{\AAA - H}$.

     We have a natural identification of 
     $\CC[\CC^n/H] \otimes \wedge(\CC^n/H)^*$ with the subspace
     $\Omega_{n-1} \subseteq \Omega_n$.
     By \eqref{eq:coloop-rho-relation-one} and
     \eqref{eq:coloop-rho-relation-two}, the 
     evaluation map
     $\widetilde{e}: \Omega_n \rightarrow \Omega_{n-1}$
     which sends $x_n, \theta_n \rightarrow 0$ induces
     a bigraded ring homomorphism
     \begin{equation}
         e: \Omega_n / I_{\AAA - H} \longrightarrow 
         \Omega_{n-1} / I_{\AAA \mid H}.
     \end{equation}
     Since $\widetilde{e}$ is surjective, so is $e$.
     Furthermore, if $\overline{L}$ is a line
     in $\CC^n/H$, there exists a line $L \neq L_0$
     in $\CC^n$ with $\pi(L) = \overline{L}$.
     Equation~\eqref{eq:coloop-rho-relation-one} implies that
     the generators 
     \begin{equation}
         \left( \lambda_{\overline{L}} \right)^{\rho_{\AAA \mid H}(\overline{L}) + 1} \quad \text{and} \quad
         d\left( \lambda_{\overline{L}} \right)^{\rho_{\AAA \mid H}(\overline{L}) + 1}
     \end{equation}
     are the images under $\widetilde{e}$ of the generators 
     \begin{equation}
         \left( \lambda_{L} \right)^{\rho_{\AAA - H}(L) + 1} \quad \text{and} \quad
         d\left( \lambda_{L} \right)^{\rho_{\AAA - H}(L) + 1}
     \end{equation}
     of $I_{\AAA - H}$.
     The map 
     $e: \Omega_n / I_{\AAA - H} \rightarrow \Omega_{n-1} / I_{\AAA \mid H}$
     is therefore an isomorphism.
 \end{proof}

We are now ready to describe the consequences of these preceding results.
 \begin{theorem}\label{th:hilb_superspace}
     For any rank $r$ multiarrangement $\AAA$ of size $m$ in $\CC^n$ we have
     \[
        \Hilb\left (\external_{\AAA};q,t\right )=\Hilb(I_{\AAA}^{\perp};q,t)=(1+t)^rq^{m-r}T_{\AAA}\left (\frac{1+q+t}{1+t},\frac{1}{q}\right )
     \]
     where $q$ tracks bosonic degree and $t$ tracks
     fermionic degree.
 \end{theorem}
 \begin{proof}
    It suffices to consider the Hilbert series for $I_{\AAA}^{\perp}$.
    
    We interpret $\Hilb(-;q,t)$ as function from the class of realizable matroids over $\CC$ (which are equivalent to multiarrangements $\AAA$ that allow for degenerate hyperplanes) to $\mathbb{N}[q,t]$ and proceed to show that it meets the criteria of a \emph{Tutte--Grothendieck invariant} as outlined in \cite{BO92}.\footnote{Note that the class of realizable matroids is a minor-closed family, which means we are justified in applying this strategy.}

    Begin by noting that if $\AAA$ is the empty arrangement, then $\Hilb(I_{\AAA}^{\perp};q,t)$ is indeed equal to $1$.
    So we may suppose that  $\AAA$ is nonempty.
    Note further that if $H$ is a degenerate hyperplane in $\AAA$ then 
    \begin{align}\label{eq:hilb_loop_deletion}
        \Hilb(I_{\AAA}^{\perp};q,t)=\Hilb(I_{\AAA-H}^{\perp};q,t).
    \end{align}
    Now suppose that $\AAA$ contains a hyperplane $H$ that is not degenerate. Then Lemma~\ref{exact-spanning-lemma}\ref{it4} tells us that 
    \begin{align}\label{eq:hilb_from_exact_sequence}
        \Hilb(I_{\AAA}^{\perp};q,t) = q \cdot \Hilb(I_{\AAA-H}^{\perp};q,t) + (1+t) \cdot \Hilb(I_{\AAA\mid H}^{\perp};q,t).
    \end{align}
    If we were to further assume that $H$ is a coloop, i.e. its normal vector belongs to every basis in $\mathfrak{M}_{\AAA}$, then Lemma~\ref{lem:coloop-coincidence} tells us that~\eqref{eq:hilb_from_exact_sequence} may be rewritten as
    \begin{align}\label{eq:hilb_for_coloop}
        \Hilb(I_{\AAA}^{\perp};q,t) =  (1+q+t) \cdot \Hilb(I_{\AAA\mid H}^{\perp};q,t).
    \end{align}
    Finally observe that the preceding equalities inductively imply that if multiarrangements $\AAA$ and $\mathcal{B}$ determine isomorphic matroids $\MMM_{\AAA}$ and $\MMM_{\mathcal{B}}$ then $\Hilb(I_{\AAA}^{\perp};q,t)=\Hilb(I_{\mathcal{B}}^{\perp};q,t)$.

    Now our claim follows from the fact that the Tutte polynomial is a \emph{universal Tutte--Grothendieck invariant}, which in particular dictates that $\Hilb$ is a specialization of the Tutte polynomials. The precise specialization can be read off immediately from  Equations~\eqref{eq:hilb_loop_deletion}--\eqref{eq:hilb_for_coloop}, and we leave the details for the reader.
 \end{proof}

 \begin{remark}
     The Hilbert series in Theorem~\ref{th:hilb_superspace} can also be rewritten in terms of the \emph{coboundary polynomial} $\overline{\chi}_{\AAA}$, which is related  to $T_{\AAA}$ by a simple change of variables (see \cite[Definition 3.2]{Ar07} and ensuing discussion). We then obtain
     \[
        \hilb(\external_{\AAA};q,t)=q^{m}\left(\frac{1+t}{1-q}\right)^r \overline{\chi}_{\AAA}\left(\frac{1-q}{1+t},\frac{1}{q}\right).
     \]
     We shall have no further use of this re-expression.
 \end{remark}

\begin{example}\label{ex:demonstrating_hilbert}
    Consider the arrangement $\AAA$ in $\mathbb{R}^2$ determined by the hyperplanes $x_1=0$, $x_2=0$, and $x_1+x_2=0$, shown on the left in Figure~\ref{fig:hyperplane_plus_generic}. 
    One may check that 
    \[
        I_{\AAA}=(x_1^3,x_2^3,(x_1-x_2)^3,x_1^2x_2^2,\theta_1x_1^2,\theta_2x_2^2, (\theta_1-\theta_2)(x_1-x_2)^2, x_1x_2(\theta_1x_2+\theta_2x_1)).
    \]
    One can then compute the following monomial basis for $\external_{\AAA}$:
    \[
    \{1, x_1, x_2, x_1^2, x_1x_2, x_2^2, x_1x_2^2,
\theta_1, \theta_2, x_1\theta_1, x_1\theta_2, x_2\theta_1, x_2\theta_2, x_1x_2\theta_1, x_1x_2\theta_2, x_2^2\theta_1,
\theta_1\theta_2, x_1\theta_1\theta_2, x_2\theta_1\theta_2\}.
    \]
    We thus see that 
    \[
    \Hilb(\external_{\AAA};q,t)=(1+2q+3q^2+q^3)+t(2+4q+3q^2)+t^2(1+2q).
    \]
    The Tutte polynomial of $\mathfrak{M}_{\AAA}$ is $T_{\AAA}(x,y)=x^2+x+y$, and it can be checked that substituting as per Theorem~\ref{th:hilb_superspace} produces this same expression.
\end{example}

It is natural to wonder whether this Tutte polynomial specialization is already present in literature.
Eur--Huh--Larson \cite[Theorem 1.11]{EHL23} recently obtained the following 4-variable specialization by way of the degree map on the stellahedral variety that they introduced:
\begin{align}\label{eq:tm_ehl}
    t_{\MMM}(x,y,z,w)=(y+z)^r(x+w)^{m-r}T_{\MMM}\left (\frac{x+y}{y+z},\frac{x+y+z+w}{x+w}\right).
\end{align}
Here $\MMM$ is any matroid on a ground set of $m$ elements and $r$ is the rank of $\MMM$. Setting $x=1$, $y=0$, $z=q$ and $w=t$ reduces the right-hand side to $q^{r}(1+t)^{m-r}T_{\MMM}\left(\frac{1}{q},\frac{1+q+t}{1+t}\right)$. 
Replacing $\MMM$ by its \emph{dual matroid} $\MMM^*$ then produces the Hilbert series in Theorem~\ref{th:hilb_superspace}.

In \cite[Theorem 1.12]{EHL23} it is shown that $t_{\MMM}$ is \emph{denormalized Lorentzian} (cf. \cite{BH20}  or \cite[Section 4.3]{BLP23}  for precise definitions). 
Given that Lorentzian-ness is preserved under linear change of variables with nonnegative coefficients \cite[Theorem 2.10]{BH20}, we may consider the consequences of our specialization in this light. Setting $y=0$, $z=q$, and $w=t$ produces a homogeneous denormalized Lorentzian polynomial. Subsequently \cite[Theorem 4.4]{BH20} allows us to infer that:
\begin{enumerate}
    \item The polynomials in $q$ (respectively $t$) obtained by considering a fixed power of $t$ (respectively $q$) in $\Hilb(\external_\AAA;q,t)$ have log-concave coefficients.
    \item The homogeneous summands of $\Hilb(\external_\AAA;q,t)$, obtained by taking all terms of a fixed total degree in $q$ and $t$, have log-concave coefficients.
\end{enumerate}
We use these observations to motivate looking closely at various families of polynomials hiding in $\Hilb(\external_{\AAA};q,t)$.

 \subsection{The classical external and central cases}
As a corollary of Theorem~\ref{th:hilb_superspace}, we next show how the superspace quotient $\external_{\AAA}$ in fact `knows' about a few classical algebraic constructs in the context of hyperplane arrangements. Recall that $r$ is the rank of the arrangement $\AAA$. Let us denote the \emph{characteristic polynomial} of $\AAA$ by $\chi_{\AAA}$.

\begin{corollary}\label{cor:hilb_cases}
    Let $r$ be the rank of the arrangement $\AAA\subseteq \CC^n$ and let $m=|\AAA|$.
    The following hold:
    \begin{enumerate}[label=(\arabic*)]
    \item \label{it5} $\Hilb(\external_{\AAA};q,0)=\Hilb(\CC[\xx_n]/J_{\AAA,1};q)$.
    \item \label{it6} $[t^r]\Hilb(\external_{\AAA};q,t)=\Hilb(\CC[\xx_n]/J_{\AAA,0};q)$ where $[t^r](-)$ extracts the coefficient 
    of $t^r$.
    \item \label{it7} Let $\operatorname{top}({\AAA})$ denote the summand of maximal total degree in $\Hilb(\external_{\AAA};q,t) \in \mathbb{N}[q,t]$. Then
    \[
       \operatorname{top}(\AAA)=(-1)^{r}q^{m-r}t^{r}\chi_{\AAA}\left(-\frac{q}{t}\right).
    \]
    \end{enumerate}
\end{corollary}
\begin{proof}
\ref{it5} is immediate so we focus on the remaining assertions. 
As before we think of $\AAA$ as a matroid on a ground set $\groundset$.
The definition of the Tutte polynomial in~\ref{eq:defn_tutte} implies that
\begin{align}
\label{eq:rewriting_bigraded_hilb}
(1+t)^{r}q^{m-r}T_{\AAA}\left(\frac{1+q+t}{1+t},\frac{1}{q}\right)
 & =
(1+t)^{r}q^{m-r}\sum_{A\subseteq \groundset}
\left( \frac{1+q+t}{1+t}-1\right)^{r-\rank{A}}\left(\frac{1}{q}-1\right)^{|A|-\rank{A}}\nonumber\\
&=(1+t)^{r}q^{m-r}\sum_{A\subseteq \groundset}
 \frac{q^{r-\rank{A}}(1-q)^{|A|-\rank{A}}}{(1+t)^{r-\rank{A}}q^{|A|-\rank{A}}}
 \nonumber\\
 &=\sum_{A\subseteq \groundset} q^{m-|A|}(1-q)^{|A|-\rank{A}}(1+t)^{\rank{A}}.
\end{align}
We thus infer that
\begin{align}
\label{eq:top_theta_deg}
    [t^r]\Hilb(\external_{\AAA};q,t) = \sum_{\substack{A\subseteq \groundset\\ \rank{A}=r}} q^{m-|A|}(1-q)^{|A|-r} &= q^{m-r}\sum_{\substack{A\subseteq \groundset\\ \rank{A}=r}} q^{r-|A|}(1-q)^{|A|-r}\nonumber\\
    & = q^{m-r}T_{\AAA}\left(1,q^{-1}\right),
\end{align}
and the right hand side in~\eqref{eq:top_theta_deg} equals $\Hilb(\CC[\xx_n]/J_{\AAA,0})$ by~\eqref{eq:hilb_classical_central}. This establishes~\ref{it6}.

From~\eqref{eq:rewriting_bigraded_hilb} it follows that the total top degree summand is given by 
\begin{align}
    \operatorname{top}(\AAA) 
    & = \sum_{A\subseteq\groundset}(-1)^{|A|-\rank{A}}q^{m-\rank{A}}t^{\rank{A}}\\
& = (-1)^{r}q^{m-r}t^r\sum_{A\subseteq \groundset }(-1)^{|A|}
\left(-\frac{q}{t}\right)^{r-\rank{A}}.
\end{align}
This last expression is essentially the expression for the characteristic polynomial for $\AAA$ when expressed as a Tutte specialization. This finishes the proof for~\ref{it7}.
\end{proof}

Going back to Example~\ref{ex:demonstrating_hilbert}
we see that the coefficients of $t^0$ and $t^2$ are indeed the Hilbert series for the classical external and central zonotopal algebras. 
As far as $\mathrm{top}(\AAA)$ is concerned we get $q^3+3q^2t+2qt^2=q(q+t)(q+2t)$, and this is clearly a homogenized version of the characteristic polynomial of the central arrangement in Figure~\ref{fig:hyperplane_plus_generic}.

\subsection{The case of real arrangements}
\label{subsec:real}

For this subsection, we suppose that the arrangement $\AAA\subset \CC^n$ is obtained by complexifying a real arrangement in $\mathbb{R}^n$ that we shall continue to call $\AAA$.
Equivalently, we may think of the matroid $\mathfrak{M}_{\AAA}$ corresponding to $\AAA$ as being realizable over $\mathbb{R}$. We are working under this assumption primarily because we would like to relate the dimension of  $I_{\AAA}^{\perp}$ (or equivalently $\external_{\AAA}$) to face enumeration in appropriate real arrangements.

Let $\AAA^{\mathrm{gen}}$ be a \emph{generic} affine arrangement determined by $\MMM$. 
Thus hyperplanes in $\AAA^{\mathrm{gen}}$ are obtained as affine translates of linear hyperplanes in $\AAA$ with the constraint that  $k$ hyperplanes have a nonempty intersection if and only if the normals to these hyperplanes form a linearly independent set.

Let $f_i(\AAA^{\mathrm{gen}})$ for $0\leq i\leq n$ denote the number of $i$-dimensional faces in the polyhedral complex on $\mathbb{R}^n$ induced by $\AAA^{\mathrm{gen}}$. For any hyperplane arrangement, an expression for these face numbers in terms of the M\"obius values of intersection poset of $\AAA^{\mathrm{gen}}$ is derived in \cite{Zas75} in terms of a specialization of the two-variable \emph{Whitney polynomial}. We prefer to follow the relatively modern (and more general) treatment in \cite[Section 6]{Ath96}.

If we ignore the bosonic grading, the algebra
$\external_\AAA$ may be regarded as a singly-graded
ring under the fermionic grading.
Up to grading reversal, the face numbers $f_i(\AAA^{\mathrm{gen}})$ form the coefficients
of the corresponding Hilbert series.

\begin{corollary}\label{cor:f_vector}
For any complexified real arrangement $\AAA$
we have
\[
\Hilb(\external_{\AAA};1,t) = \sum_{0\leq i\leq n}f_{n-i}(\AAA^{\mathrm{gen}}) \cdot t^i.
\]
\end{corollary}

Observe that $f_n$ and $f_0$ give the dimensions of the external and central zonotopal algebras respectively and this is in agreement with Corollary~\ref{cor:hilb_cases}\ref{it5}--\ref{it6}.

\begin{proof}
Throughout this proof we let $L$ denote the intersection poset of $\AAA^{\mathrm{gen}}$ and $\mu$ its M\"{o}bius function. 
By \cite[Theorem 6.4]{Ath96} and the discussion leading up to it we have 
\begin{align}\label{eq:whitney}
    \sum_{0\leq i\leq n}f_{n-i}(\AAA^{\mathrm{gen}}) \cdot t^i=(-1)^{n}\sum_{x\leq_{L} z}\mu(x,z)(-t)^{n-\dim(x)}(-1)^{\dim(z)}.
\end{align}

Recall that the intersection poset of a generic arrangement is isomorphic to a truncated Boolean lattice. 
Furthermore, observe additionally 
that if we let $\mc{I}$ denote the set of independent subsets of the matroid $\mathfrak{M}_{\AAA}$, then the \emph{independence complex}\footnote{the simplicial complex underlying $\mc{I}$ interpreted as a poset with partial order given by inclusion}  is also isomorphic to $L$.
Hence we may rewrite~\eqref{eq:whitney} as
\begin{align}\label{eq:new_whitney}
    \sum_{0\leq i\leq n}f_{n-i}(\AAA^{\mathrm{gen}}) \cdot t^i
    & = (-1)^{n}\sum_{\substack{B\subseteq A\\A,B\in \mc{I}}}(-1)^{|B|-|A|}(-t)^{|B|}(-1)^{n-|A|}\nonumber \\
    & = \sum_{\substack{B\subseteq A\\A,B\in \mc{I}}} t^{|B|}\nonumber \\
    & = \sum_{A\in \mc{I}} (1+t)^{|A|}.
\end{align}
This last expression is precisely $\hilb(\external_{\AAA};q,t)\mid_{q = 1}$ as can be verified by setting $q=1$ in \eqref{eq:rewriting_bigraded_hilb}.
\end{proof}

\begin{figure}[!h]
    \includegraphics[scale=1]{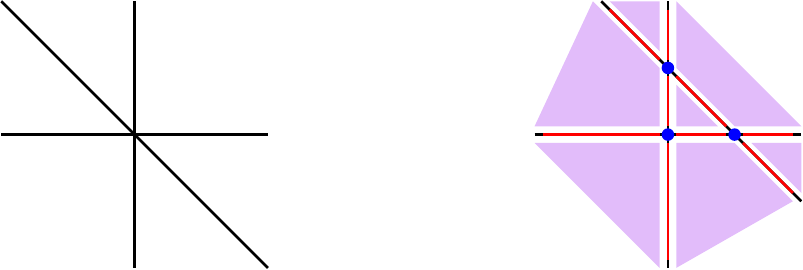}
    \caption{A central arrangement (left) and a generic deformation thereof (right). The polyhedral complex induced has 7 full-dimensional faces, nine $1$-dimensional faces, and three $0$-dimensional faces.}
    \label{fig:hyperplane_plus_generic}
\end{figure}

\begin{remark}
\label{rmk:fields}
    Corollary~\ref{cor:f_vector}
    bears similarity to a result in 
    superspace coinvariant theory.
    Let $\AAA_n^\braid$ be the type A$_{n-1}$
    braid arrangement in $\RR^n$
    with hyperplanes $x_i - x_j = 0$
    for $1 \leq i < j \leq n$.
    Let $f_i(\AAA_n^\braid)$ be the number 
    of $i$-dimensional faces in $\AAA_n^\braid$, so that 
    $f_i(\AAA_n^\braid)$ counts ordered set partitions
    of $[n]$ into $i$ blocks.

    The symmetric group $\symm_n$ acts 
    on both $\CC[\xx_n]$ and $\Omega_n$ by permuting 
    subscripts. The {\em coinvariant ideal}
    $I_n \subseteq \CC[\xx_n]$ is generated by 
    $\symm_n$-invariants with vanishing constant term.
    Explicitly, we have 
    $I_n = (e_1, \dots, e_n)$ where $e_j \in \CC[\xx_n]$
    is the elementary symmetric polynomial of degree $j$.
    The {\em coinvariant ring}
    $\CC[\xx_n]/I_n$ has vector space dimension
    $n!$, the number of maximal-dimensional 
    faces in $\AAA_n^\braid$.
    Define $SI_n \subseteq \Omega_n$ to be the 
    differential closure of $I_n$, so that
    $SI_n = (e_1, \dots, e_n, d e_1, \dots, d e_n)$.
    The ideal $SI_n$ is generated by $\symm_n$-invariants
    in $\Omega_n$ with vanishing constant term.
    The first and third authors proved 
    \cite{rhoades2023hilbert} that 
    \begin{equation}
        \Hilb(\Omega_n/SI_n; 1, t) = 
        \sum_{i = 0}^{n-1} f_{n-i}(\AAA^\braid_n) \cdot t^i
    \end{equation}
    so that the {\em superspace coinvariant ring}
    $\Omega_n/SI_n$ coming from the differential
    closure of $I_n$
    encodes information about faces of higher 
    codimension. For example, when $n = 3$
    we have
    \[
    \Hilb(\Omega_3/SI_3;1,t) = 6 + 6 t + t^2
    \]
    corresponding to the decomposition of $\RR^3$
    by $\AAA^\braid_3$ into 6 faces of codimension 0,
    6 faces of codimension 1, and 1 face of codimension 2.
    
    It would be interesting to see how far the analogy
    between superization and faces of higher
    codimension can be carried in the context of 
    arrangements.
    The Weyl arrangement of type F$_4$ shows that there
    will be subtlety here.
    The fermionic Hilbert series 
    of the quotient of $\Omega_4$ by the 
     F$_4$-invariants ideal is
    $1152 t^0 + 2304 t^1 + 1396 t^2 + 244 t^3  + t^4$
    whereas the reversed $f$-vector coming from
     the F$_4$ Weyl arrangement is slightly different:
     $(1152, 2304, 1392, 240, 1)$.
\end{remark}

We now describe another interpretation for the dimension of $\external_{\AAA}$ which casts a different light.
Given a matroid $\MMM$ and $d \geq 1$, let $d\MMM$ denote its \emph{$d$-fold thickening}, i.e. the matroid obtained obtained by including $(d-1)$ additional parallel elements for each element in $\MMM$. The Tutte polynomial of  $d\MMM$ is known (cf. \cite[Lemma 4.4]{Ber10}) to satisfy
\begin{align}
T_{d\MMM(x,y)} = \left(\frac{1-y^d}{1-y}\right)^{r}T_{\MMM}\left(\frac{xy-x-y-y^d}{y^d-1},y^d\right).
\end{align}
At $d=2$ this becomes
\begin{align}
\label{eq:doubled_tutte}
T_{2\MMM}(x,y)=(1+y)^{r}T_{\MMM}\left(\frac{x+y}{y+1},y^2\right).
\end{align}

Given a central multiarrangement $\AAA\subseteq \mathbb{R}^n$, 
let us abuse notation and denote the two-fold thickening of the underlying matroid by $2\AAA$. 
We may now consider the generic hyperplane arrangement $2\AAA^{\mathrm{gen}}$.
The following result makes a curious connection between the dimension of the classical external zonotopal algebra attached to $2\AAA$ to the superspace analogue attached to $\AAA$. More generally we have:
\begin{corollary}\label{cor:doubled_generic}
Let $\AAA\subseteq \mathbb{R}^n$ be a central multiarrangement.
The following equality holds:
\[
    \Hilb(\external_{\AAA};q,t)|_{(q,t)=(q^2,q)}=\Hilb(\CC[\xx_n]/J_{2\AAA,1};q).
\]
Thus $\dim(\external_{\AAA})$ equals the number of regions in $2\AAA^{\mathrm{gen}}$.
\end{corollary}
\begin{proof}
Let $r$ denote the rank of the arrangement $\AAA$ and let $m=|\AAA|$.
By combining~\ref{eq:hilb_classical_external} and~\ref{eq:doubled_tutte} we have
\begin{align}
\Hilb(\CC[\xx_n]/J_{2\AAA,1};q)=q^{2m-r}\,T_{2\AAA}\left(1+q,\frac{1}{q}\right)
& = (q^2)^{m-r}(1+q)^r\,T_{\mc{M}}\left(\frac{1+q+q^2}{1+q},\frac{1}{q^2}\right)\nonumber\\
& = \Hilb(\external_{\AAA};q,t)|_{(q,t)=(q^2,q)},
\end{align}
where the last equality follows upon comparison with Theorem~\ref{th:hilb_superspace}.

To get the second half of the claim, note that setting $q=1$ in $q^{2m-r}T_{2\AAA}\left(1+q,\frac{1}{q}\right)$ gives the number of independent sets in the collection of vectors determining $2\mathfrak{M}_{\AAA}$, which in turn is equal to the number of regions in $2\AAA^{\mathrm{gen}}$.
\end{proof}

As an example we consider in Figure~\ref{fig:hyperplanes_doubled_generic} the generic hyperplane arrangement in $\mathbb{R}^2$ obtained by `doubling' each hyperplane in the arrangement in Figure~\ref{fig:hyperplane_plus_generic}.
The reader is welcome to verify that this arrangement has 19 regions, and this matches the sum of the $f$-vector coefficients.

\begin{figure}[!h]
    \includegraphics[scale=0.8]{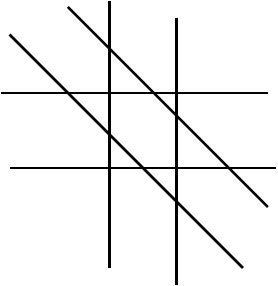}
    \caption{A generic arrangement corresponding to a two-fold thickening of the central arrangement in Figure~\ref{fig:hyperplane_plus_generic}.}
    \label{fig:hyperplanes_doubled_generic}
\end{figure}

\subsection{The case of graphical arrangements}
\label{subsec:graphical}
We briefly dwell on a class of real arrangements that has been particularly well studied. Let $G$ be a simple graph with vertex set identified with $[n]\coloneqq \{1,\dots,n\}$ and $E(G)$ denoting its set of edges.  
For convenience let us assume that $    G$ is connected. Let $g=|E(G)|-n+1$ denote the \emph{genus} of $G$.
The arrangement $\AAA_G$ in $\mathbb{R}^n$ is obtained by considering the collection of hyperplanes $x_i-x_j=0$ for $\{i,j\}\in E(G)$. 
Note that the connectedness implies that the rank of $\AAA_G$ is equal to $n-1$.
Theorem~\ref{th:hilb_superspace} then says that
\begin{align}\label{eq:above}
    \Hilb(\external_{\AAA_G};q,t)=(1+t)^{n-1}q^{g}\,T_G\left(\frac{1+q+t}{1+t},\frac{1}{q} \right).
\end{align}
We note here that Backman--Hopkins \cite{BH17} in their study of \emph{fourientations} consider a family of Tutte specializations of the form $(k+l)^{n-1}(k+m)^g\,T_G\left(\frac{\alpha k+\beta l+m}{k+m},\frac{\gamma k+1+\delta m}{k+l}\right)$, where $\alpha, \gamma \in \{0,1,2\}$ and $\beta,\delta \in \{0,1\}$.
Upon setting $k=q-1$, $l=1$, $m=2+t-q$, $\gamma=0$, $\alpha=2$, $\beta=1$, $\gamma=\delta=0$ we recover the specialization in~\eqref{eq:above}. Thus Theorem~\ref{th:hilb_superspace} realizes algebraically one of the specializations considered in \cite{BH17}, and thus addresses a question in \cite[Remark 5.6]{BHT18} in a very special case.

At $q=t=1$ we recover the quantity $2^{n-1}T_G(3/2,1)$. We know by Corollary~\ref{cor:doubled_generic} that this equals the number of regions in $2\AAA_G^{\mathrm{gen}}$.
In fact this quantity has already been investigated by Hopkins--Perkinson \cite[Theorem 3.2]{HoPe16} where they call $2\AAA_G^{\mathrm{gen}}$ the \emph{generic bigraphic} arrangement. 
This evaluation at $q=1$ of the second quantity appearing in Corollary~\ref{cor:doubled_generic} is the number of lattice points in the twice-dilated graphical zonotope attached to $G$. 
This is a consequence of a more general fact (cf. \cite{HR11}) that, given a realizable matroid  corresponding to a \emph{totally unimodular} matrix $A$, the dimension of the associated external zonotopal algebra is the number of lattice points in the zonotope corresponding to $A$.
Backman--Baker--Yuen explain this numerological coincidence via a \emph{geometric bijection} in \cite[Remark 5.1.3]{BBY19}.

\section{A basis for the Macaulay inverse}
\label{sec:basis}

Fix a total order on the hyperplanes of $\AAA$; this induces a lexicographical order on the subsets of $\AAA$.
Given a basis $B$ of $\AAA$ and a hyperplane $H \in \AAA - B$, the hyperplane $H$ is {\em externally active with respect to $B$}
if $B$ is the lexicographically largest basis contained in $B \cup H$.  
Otherwise, the hyperplane $H$ is {\em externally passive with respect to $B$}.
Similarly, a hyperplane $H \in B$ is {\em internally active with respect to $B$} if $B$ is the lexicographically smallest basis containing $B - H$.
Otherwise $H$ is {\em internally passive with respect to $B$}.

By Theorem~\ref{th:hilb_superspace} we have that
\begin{align}
\label{eq:bigraded_hilb}
\hilb(I_{\AAA}^{\perp};q,t)
&=
\sum_{B\in \mc{B}}(1+q+t)^{\intact(B)}(1+t)^{\intpas(B)}q^{\extpas(B)},
\end{align}
and this suggests properties that a potential basis for $I_{\AAA}^{\perp}$ might satisfy.

\begin{example}\label{ex:activities}
Consider the hyperplane arrangement from Example~\ref{ex:demonstrating_hilbert} with the normal vectors to the hyperplanes $H_1=\{x_1=0\}$, $H_2=\{x_2=0\}$, and $H_3=\{x_1+x_2=0\}$ recorded as columns of a $2\times 3$ matrix below. Assume our total order on hyperplanes in $\AAA$ is obtained by reading the columns left to right, i.e. $H_1<H_2<H_3$.
\[
\left[\begin{array}{ccc}1 & 0 & 1\\ 0 & 1 & 1\end{array}\right]
\]
For the basis $\{H_1,H_2\}$, the hyperplane $H_3$ is externally passive whereas $H_1$ and $H_2$ are internally active.
For the basis $\{H_1,H_3\}$, the hyperplane $H_2$ is externally passive, $H_1$ is internally active while $H_3$ is internally passive. Finally for the basis $\{H_2,H_3\}$, the hyperplane $H_1$ is externally active whereas both $H_2$ and $H_3$ are internally passive. Thus the right-hand side of~\eqref{eq:bigraded_hilb} equals
\[
    (1+q+t)^2q^1(1+t)^0+(1+q+t)^1q^0(1+t)^1+(1+q+t)^0q^0(1+t)^2,
\]
which the reader may verify is the same as the Hilbert series obtained earlier.
\end{example}

 Our next goal is to give a basis for $I_\AAA^{\perp}$ which may be extracted from the spanning set of 
 Lemma~\ref{exact-spanning-lemma}\ref{it3}. Our main tool is the short exact sequence of Lemma~\ref{exact-spanning-lemma}\ref{it4}.  We  need the following 
 simple fact about short exact sequences of vector spaces.
 
 \begin{lemma}
 \label{short-exact-bases}
 Let $0 \rightarrow U \xrightarrow{ \, \, \varphi \, \, } V \xrightarrow{ \, \, \psi \, \, } W \rightarrow 0$ be a short exact sequence of vector spaces and suppose we have subsets 
 $A \subseteq U, B \subseteq V,$ and $C \subseteq W$ which satisfy the following properties.
 \begin{itemize}
 \item $A$ is a basis of $U$,
 \item $C$ is a basis of $W$, and
 \item the set $B$ may be expressed as a union $B = B' \cup B''$ of $B$ such that 
 $$
 B' = \{ \varphi(a) \,:\, a \in A \} 
 \quad \quad \text{and} \quad \quad
 C = \{ \psi(b'') \,:\, b'' \in B''\}.
 $$
 \end{itemize}
 Then $B$ is a basis of $V$, and the union $B' \sqcup B''$ is disjoint.
 \end{lemma}

Recall from the introduction that we write $EA_\AAA(B), EP_\AAA(B), IA_\AAA(B),$ and $IP_\AAA(B)$ for the externally active, externally passive, internally active, and internally passive hyperplanes in $\AAA$ with 
respect to a given basis $B$.   
Note also that if we have a total order on the hyperplanes in an arrangement $\AAA$, then the deleted and restricted arrangement $\AAA-H$ and $\AAA\mid H$ both inherit total~orders.

\begin{lemma}
\label{active-passive-decomposition}
Let $H_0 \in \AAA$ be the largest hyperplane with respect to the given total order on $\AAA$.  Assume that $H_0$ is neither a loop nor a coloop. 
Let $\BBB_\AAA$, $\BBB_{\AAA - H_0}$ and $\BBB_{A \mid H_0}$ denote the sets of bases of $\AAA$, $\AAA - H_0$, and $\AAA \mid H_0$. Write
$\BBB_\AAA = \BBB'_\AAA \sqcup \BBB''_\AAA$ where
$$
\BBB'_\AAA \coloneqq \{ B' \in \BBB_\AAA \,:\, H_0 \notin B' \} \quad \quad \text{and} \quad \quad
\BBB''_\AAA \coloneqq \{ B'' \in \BBB_\AAA \,:\, H_0 \in B'' \}.
$$
We have bijections $f: \BBB'_\AAA \rightarrow \BBB_{\AAA - H_0}$ and $g: \BBB''_\AAA \rightarrow \BBB_{\AAA \mid H_0}$ given by
$$
f(B') \coloneqq B' \quad \quad \text{and} \quad \quad g(B'') \coloneqq \{ H \cap H_0 \,:\, H \in B'', \, \, H \neq H_0 \}.
$$
Furthermore, if $B' \in \BBB'_\AAA$ and $B'' \in \BBB''_\AAA$ then $H_0$ is internally passive in $B''$ and the activities of these bases are related by
$$
\begin{cases}
EA_\AAA(B') = EA_{\AAA - H_0}(f(B')) \\
EP_\AAA(B') = EP_{\AAA - H_0}(f(B')) \cup H_0 \\
IA_\AAA(B') = IA_{\AAA - H_0}(f(B')) \\
IP_\AAA(B') = IP_{\AAA - H_0}(f(B'))
\end{cases}  \quad \quad
\begin{cases}
EA_{\AAA \mid H_0}(g(B'')) = \{ H \cap H_0 \,:\, H \in EA_\AAA(B'') \} \\
EP_{\AAA \mid H_0}(g(B'')) = \{ H \cap H_0 \,:\, H \in EP_\AAA(B'') \} \\
IA_{\AAA \mid H_0}(g(B'')) = \{ H \cap H_0 \,:\, H \in IA_\AAA(B'') \} \\
IP_{\AAA \mid H_0}(g(B'')) = \{ H \cap H_0 \,:\, H \in IP_\AAA(B''), \, \, H \neq H_0 \}. \\
\end{cases}
$$
\end{lemma}

\begin{proof}
Since $H_0$ is not a coloop of $\AAA$, the deleted arrangement $\AAA - H_0$ has the same rank as $\AAA$. A basis of $\AAA - H_0$ is therefore
precisely a basis of $\AAA$ which does not contain $H_0$, so that $f$ is a bijection. Since $H_0$ is not a loop at $\AAA$, the set $\BBB''_\AAA$ is non-empty
and the map $g$ is also a bijection.

The statements regarding activities make use of the fact that $H_0$ is the largest hyperplane in $\AAA$ and neither a loop nor a coloop. For example, given 
a basis $B'' \in \BBB''_\AAA$ of $\AAA$ which contains $H_0$, since $H_0$ is not a coloop there is more than one basis of 
$\AAA$ containing $B'' - H_0$, and since $H_0$
the largest hyperplane, $B''$ is the lexicographically largest (not smallest) basis of $\AAA$ containing $B'' - H_0$. We conclude that $H_0 \in IP_\AAA(B'')$.

Let $B' \in \BBB'_{\AAA - H_0}$ be a basis of $\AAA$ which does not contain $H_0$. Let $H \in \AAA - B'$ be a hyperplane which is not contained in $B'$, so that 
$B' \cup H$ is dependent. If $H \neq H_0$, the circuit $C$ in $B' \cup H$ containing $H$ is the same subset of $\AAA$ or of $\AAA - H_0$, so whether or not 
$H$ is minimal in $C$ is unaffected by the deletion of $H_0$. Since $H_0$ is not a loop, the circuit $C_0$ in $B' \cup H_0$ does not consist of $H_0$ alone,
and since $H_0$ is largest, we know that $H_0$ is not the minimal element of $C_0$.  We conclude that $H_0 \in EP_\AAA(B')$.
In summary, we have $EA_\AAA(B') = EA_{\AAA - H_0}(f(B'))$ and $EP_\AAA(B') = EP_{\AAA - H_0}(f(B')) \cup H_0$.

Let $B' \in \BBB'_{\AAA - H_0}$ be a basis of $\AAA$ which does not contain $H_0$ and let $H \in B'$. Then $B' - H$ is an independent set, but not a basis of $\AAA$
or of $\AAA - H_0$.
Let $D \subseteq \AAA$ be the family of hyperplanes $$D \coloneqq \{ H_1 \in \AAA \,:\, (B' - H) \cup H_1 \text{ is a basis of $\AAA$} \}.$$
The hyperplane $H_0$ may or may not be an element of $D$; let $D' \coloneqq D - H_0$ be the subset of $D$ formed by removing $H_0$ if it is present.
Since $H < H_0$ and $H \in D, D'$ we have $\min(D) = \min(D')$. We conclude that 
$IA_\AAA(B') = IA_{\AAA - H_0}(f(B'))$ and 
$IP_\AAA(B') = IP_{\AAA - H_0}(f(B'))$.

Now let $B'' \in \BBB''_\AAA$ be a basis of $\AAA$ which contains $H_0$.  Let $H \in \AAA - B''$ be a hyperplane which is not contained in $B''$.
Suppose first that $H \in EP_\AAA(B'')$. Then $H$ is not the smallest element of the fundamental circuit $C$ of $B'' \cup H$.  In particular, the hyperplane $H$ cannot
be a loop in $\AAA$ and cannot
be `parallel to' (i.e. another copy of) the largest hyperplane $H_0$ so that $H \cap H_0$ is not a loop in the restriction $\AAA \mid H_0$.
We conclude that
$$
C \mid H_0 \coloneqq \{ H' \cap H_0 \,:\, H' \in C, \, \, H' \neq H_0 \}
$$
is a dependent subset of $\AAA \mid H_0$ contained in $g(B'') \cup (H \cap H_0)$ which has at least one element smaller than $H \cap H_0$.
We claim that the dependent set $C \mid H_0$ is in fact a circuit in $\AAA \mid H_0$; this will show that $H \cap H_0 \in EP_{\AAA \mid H_0}(g(B''))$.
Indeed, choose $H_1 \in C$ with $H_1 \neq H_0$. 
We argue that $ \{ H' \cap H_0 \,:\, H' \in C, \, \, H' \neq H_0, \, H_1 \} $ is independent in $\AAA \mid H_0$. Indeed, a dependence of 
 $ \{ H' \cap H_0 \,:\, H' \in C, \, \, H' \neq H_0, \, H_1 \} $ in $\AAA \mid H_0$ would induce a dependence of $C - H_1$ in $\AAA$, which contradicts the fact that 
 $C$ is a circuit in $\AAA$.
We conclude that $C \mid H_0$ is a circuit in $\AAA \mid H_0$ so that  $H \cap H_0 \in EP_{\AAA \mid H_0}(g(B''))$.

Let $B'' \in \BBB''_\AAA$ be a basis of $\AAA$ which contains $H_0$ and $H \in \AAA - B''$ be a hyperplane which is not contained in $B''$.
Assume now that $H \in EA_\AAA(B'')$. Let $C \subseteq \AAA$ be the unique circuit contained in $B'' \cup H$, so that $H$ is the smallest element of $C$.
By the same reasoning as in the previous paragraph, the set $C \mid H_0$ is the unique circuit contained in $g(B) \cup (H \cap H_0)$.
Since the smallest element of $C$ is $H \neq H_0$, the smallest element of $C \mid H_0$ is $H \cap H_0$ and $H \cap H_0 \in EA_{\AAA \mid H_0}(g(B''))$.

Let $B'' \in \BBB''_\AAA$ be a basis of $\AAA$ which contains $H_0$ and $H \in B''$ be a hyperplane in $B''$ with $H \neq H_0$.
Since $H$ is an element of the basis $B''$, we know that $H$ is not a loop in $\AAA$.
Since the basis $B''$ contains $H_0$, we also know that $H$ is not parallel to $H_0$, so that $H \cap H_0$ is not a loop in $\AAA \mid H_0$.
As before, let $D =  \{ H_1 \in \AAA \,:\, (B'' - H) \cup H_1 \text{ is a basis of $\AAA$} \}$ and consider the subset
$$
D'' \coloneqq \{ H_1 \cap H_0 \,:\, H_1 \in D \} \subseteq \AAA \mid H_0
$$
of $\AAA \mid H_0$ obtained by intersecting each hyperplane in $D$ with $H_0$. Since $H_0 \in (B'' - H)$, each element of $D''$ is a genuine hyperplane in $H_0$. By linear algebra
we have
$$
D'' = \{ H_1 \cap H_0 \in \AAA \mid H_0 \,:\, (g(B'') - (H \cap H_0)) \cup (H \cap H_1) \text{ is a basis of $\AAA \mid H_0$ } \}.
$$
The hyperplane $H$ is minimal in $D$ if and only if the intersection $H \cap H_0$ is minimal in $D''$, so that 
$$
H \in IA_\AAA(B'') \quad \Leftrightarrow \quad 
H \cap H_0 \in IA_{\AAA \mid H_0}(g(B''))
$$
and
$$
H \in IP_\AAA(B'') \quad \Leftrightarrow \quad 
H \cap H_0 \in IP_{\AAA \mid H_0}(g(B'')).
$$
Since $H_0 \in IP_{\AAA}(B'')$, this completes the proof.
\end{proof}

We now consider the family $M_\AAA$ of superspace elements given by
\begin{equation}
M_\AAA \coloneqq \bigcup_B
\left\{
\prod_{e \, \in \, E} \alpha_e \times \prod_{i \, \in \, I} d \alpha_i \times \prod_{s \, \in \, S} \alpha_s \times \prod_{t \, \in \, T} d \alpha_t \,:\, 
\begin{array}{c}
E = EP_\AAA(B), \, \,
I \subseteq IP_\AAA(B), \\ S, T \subseteq IA_\AAA(B), \, \, S \cap T = \varnothing
\end{array} 
\right\}
\end{equation}
where the union is over bases $B$ of $\AAA$. 
Observe that every superpolynomial of $M_\AAA$ appears in statement of Lemma~\ref{perp-membership}, and 
therefore lies in the spanning set of $I_\AAA^{\perp}$ coming from Lemma~\ref{exact-spanning-lemma}.
Even better, we have 
\begin{theorem}
\label{m-is-basis}
The set $M_\AAA$ forms a basis for $I_\AAA^{\perp}$.
\end{theorem}

\begin{proof}
If $\AAA$ consists entirely of loops and coloops, this is a consequence of Lemma~\ref{coordinate-case}.
We may therefore assume that $\AAA$ contains a hyperplane $H_1$ which is neither a loop nor a coloop, and that $H_1$ is the largest hyperplane
in $\AAA$. We may also inductively assume that the theorem has been established for the deleted arrangement $\AAA - H_1$ and the 
restricted arrangement $\AAA \mid H_1$.

Changing coordinates if necessary, we may assume that the hyperplane $H_1$ is $\{ x_1 = 0 \}$ so that $\alpha_{H_1} = x_1$ and $d \alpha_{H_1} = \theta_1$.
By Lemma~\ref{exact-spanning-lemma}, we have an exact sequence
 \begin{equation}
 0 \rightarrow I^{\perp}_{\AAA - H_1}  \xrightarrow{ \, \, \varphi \, \, } I^{\perp}_\AAA \xrightarrow{ \, \, \psi \, \,} I^{\perp}_{\AAA \mid H_1} \oplus I^{\perp}_{\AAA \mid H_1} \rightarrow 0
 \end{equation}
 where $\varphi(f) = x_1 \cdot f$ and $\psi(g) = \left( g \mid_{x_1, \theta_1 = 0}, \theta_1 \odot g \mid_{x_1 = 0} \right)$. 
 By Lemma~\ref{active-passive-decomposition},
the set $M_\AAA$ of superspace elements decomposes as a union 
\begin{equation}
M_\AAA = M^{\varnothing}_\AAA \cup M^{\xx}_\AAA \cup M^{\ttheta}_\AAA
\end{equation}
of three subsets. Here
\begin{multline}
M^{\xx}_\AAA = x_1 \cdot M_{\AAA - H_1} = \\
\bigcup_{B' \, \in \, \BBB'}
\left\{
\prod_{e \, \in \, E} \alpha_e \times \prod_{i \, \in \, I} d \alpha_i \times \prod_{s \, \in \, S} \alpha_s \times \prod_{t \, \in \, T} d \alpha_t \,:\, 
\begin{array}{c}
E = EP_\AAA(B'), \, \,
I \subseteq IP_\AAA(B'), \\ S, T \subseteq IA_\AAA(B'), \, \, S \cap T = \varnothing
\end{array} 
\right\}
\end{multline}
and $\BBB'$ denotes the family of bases $B'$ of $\AAA$ which do not contain $H_0$,
and if $\BBB''$ denotes the bases $B''$ of $\AAA$ which contain $H_0$ we have
\begin{equation}
M^{\varnothing}_\AAA =  \\
\bigcup_{B'' \, \in \, \BBB''}
\left\{
\prod_{e \, \in \, E} \alpha_e \times \prod_{i \, \in \, I} d \alpha_i \times \prod_{s \, \in \, S} \alpha_s \times \prod_{t \, \in \, T} d \alpha_t \,:\, 
\begin{array}{c}
E = EP_\AAA(B''), \, \,
I \subseteq IP_\AAA(B''), \\  H_1 \notin I \\ S, T \subseteq IA_\AAA(B''), \, \, S \cap T = \varnothing
\end{array} 
\right\}
\end{equation}
and
\begin{equation}
M^{\ttheta}_\AAA = \theta_1 \cdot M^{\varnothing}_\AAA.
\end{equation}

Observe that $M^{\xx}_\AAA = x_1 \cdot M_{\AAA - H_1}$ is the image of $M_{\AAA - H_1}$ under $\varphi$; by induction,
the set $M_{\AAA - H_1}$ is a basis of $I_{\AAA - H_1}^{\perp}$.
We claim that $\psi$ carries $M^{\varnothing}_\AAA \cup M^{\ttheta}_\AAA$
onto a basis of $I_{\AAA \mid H_1}^{\perp} \oplus I_{\AAA \mid H_1}^{\perp}$.  Indeed, by Lemma~\ref{active-passive-decomposition} and the inductively
established version of the theorem for $I_{\AAA \mid H_1}^{\perp}$, we know that 
$\{ f \mid_{x_1, \theta_1 = 0} \,:\, f \in M^{\varnothing}_\AAA \}$ is a basis of $I^{\perp}_{\AAA \mid H_1}$. Given any element $f \in M^{\varnothing}_\AAA$,
we have a corresponding element $\theta_1 \cdot f \in M^{\ttheta}_\AAA$. Calculating the images $\psi(f)$ and $\psi(\theta_1 \cdot f)$ of these elements under
$\psi$ gives
\begin{align}
\psi(f) &= ( f \mid_{x_1, \theta_1 = 0}, \theta_1 \odot f \mid_{x_1 = 0} ), \\
\psi(\theta_1 \cdot f) &= ( 0, f \mid_{x_1, \theta_1 = 0} ).
\end{align}
Taken together, these equations imply that $\psi \left( M^{\varnothing}_\AAA \cup M^{\ttheta}_\AAA \right)$ is indeed a basis of $I^{\perp}_\AAA \oplus I^{\perp}_\AAA$.
By Lemma~\ref{short-exact-bases}, we conclude that $M_\AAA = M^{\xx}_\AAA \cup \left( M^{\varnothing}_\AAA \cup M^{\ttheta}_\AAA \right)$ is a basis
of $I^{\perp}_\AAA$, and the theorem is~proven.
\end{proof}


\section{Further remarks}
\label{sec:fin}

\subsection{Interpolating between the classical internal and central cases}
\label{subsec:internal}
The reader may wonder why the internal zonotopal algebra does not make an appearance in this work. We clarify this matter by putting forth some conjectures and the hurdles that one would need to surmount.

Consider the superspace ideal
\begin{equation}
I'_{\AAA}\coloneqq I_{\AAA,0} \coloneqq \left( \lambda_L^{\rho_\AAA(L)}, d \lambda_L^{\rho_\AAA(L)} \,:\, L \subseteq \CC^n \text{ a line}  \right) \subseteq \Omega_n.
\end{equation}
This ideal is the differential closure of the ideal $J_{\AAA,0}$ according to the definition in~\eqref{eq:power_ideals}. 
Part of the reason why the quotient $\Omega_n/I'_\AAA$ might need an even more intricate analysis stems from the fact that it sees both the classical central zonotopal algebra as well as the internal zonotopal algebra, as explained after Conjecture~\ref{conj:internal}.
It is this latter algebra that is comparatively (and surprisingly) difficult to understand. 
Even though its Hilbert series is known and, keeping with the theme, equals a specialization of the Tutte polynomial (cf. \cite{HR11} or \cite[Proposition 4.15]{AP10}), there is no known basis for the Macaulay inverse. Indeed, as pointed out in \cite{APcorr} in the disproof of a conjecture of Holtz--Ron, the proof of the claim in \cite[Proposition 4.5(iii)]{AP10} is faulty. 
See also the discussion after~\cite[Theorem 4.2]{Be18} on this front.

We put forth the following conjecture.
\begin{conjecture}\label{conj:internal}
    Let $\AAA$ be a rank $r$ central multiarrangement in $\CC^n$ with $|\AAA|=m$.
    The bigraded Hilbert series of $\Omega_n/I'_{\AAA}$ satisfies the equality
    \[
    \Hilb\left(\Omega_n/I'_{\AAA};q,t\right)=(1+t)^rq^{m-r}T_{\AAA}\left(\frac{1}{1+t},\frac{1}{q}\right).
    \]
\end{conjecture}
At $t=0$ we recover $\Hilb(\CC[\xx_n]/J_{\AAA,0};q)$ as is expected. Developing the right-hand side in Conjecture~\ref{conj:internal} following~\eqref{eq:defn_tutte} gives 
\begin{align}
    (1+t)^rq^{m-r}T_{\AAA}\left(\frac{1}{1+t},\frac{1}{q}\right)& = q^{m-r}\sum_{A\subseteq \mathsf{E}}(-1)^{r-r(A)} t^{r-r(A)}(1+t)^{r(A)}\left(\frac{1}{q}-1\right)^{|A|-r(A)},
\end{align}
and extracting the coefficient of $t^r$ gives $q^{m-r}T_{\AAA}\left(0,q^{-1}\right)$. This last quantity is also the singly graded Hilbert series of the internal zonotopal algebra $\CC[\xx_n]/J_{\AAA,-1}$. Thus the truth of the preceding conjecture implies that an exact parallel to the first two claims in Corollary~\ref{cor:hilb_cases} holds.
With some more work, the reader may convince themselves that an analogue of Corollary~\ref{cor:f_vector} holds for real arrangements as well -- the only change being one now records the $f$-vector of the `bounded' polyhedral complex of $\AAA^{\mathrm{gen}}$.

\begin{example}\label{ex:demonstrating_conj_hilbert}
    Consider $\AAA\subseteq \mathbb{R}^2$ determined by the hyperplanes $x_1=0$, $x_2=0$, and $x_1+x_2=0$.
    One may check that 
    \[
        I_{\AAA}=(x_1^2,x_2^2,x_1x_2,\theta_1x_1,\theta_2x_2, \theta_1x_2+\theta_2x_1).
    \]
    One can then compute the following monomial basis for $\Omega_n/I'_{\AAA}$:
    \[
    \{1, x_1, x_2, \theta_1,\theta_2,\theta_1x_2,\theta_1\theta_2\}.
    \]
    We thus see that 
    \[
    \Hilb(\Omega_n/I'_{\AAA};q,t)=(1+2q)+t(2+q)+t^2.
    \]
    It is easily checked that $(1+t)^2qT_{\AAA}(1/(1+t),1/q)$ for  $T_{\AAA}(x,y)=x^2+x+y$ is precisely the Hilbert series computed above, in harmony with Conjecture~\ref{conj:internal}.
    Additionally, note that setting $q=1$ gives the polynomial $3+3t+t^2$ whose sequence of coefficients agrees with the $f$-vector corresponding to the unique bounded face of the arrangement on the right in Figure~\ref{fig:hyperplane_plus_generic}. 
\end{example}

Recall also from Section~\ref{sec:spanning} that we described $\Hilb(\external_{\AAA};q,t)$ by a specialization of a formula due to Eur--Huh--Larson \cite{EHL23}. The expression in Conjecture~\ref{conj:internal}, on the other hand, is not obtained using their results. 
Curiously, it is the earlier work of Berget--Eur--Spink--Tseng \cite[Theorems A and B]{BEST23} that we need instead. Using the degree map on the permutahedral toric variety the aforementioned set of authors consider the specialization
\begin{align}\label{eq:tm_best}
    \underline{t}_{\MMM}(x,y,z,w)\coloneqq (x+y)^{-1}(y+z)^{r}(x+w)^{n-r}T_{\MMM}\left(\frac{x+y}{y+z},\frac{x+y}{x+w}\right),
\end{align}
where $r$ is the rank of the matroid $\MMM$ whose ground set has cardinality $n$. 
Setting $x=0$, $y=1$, $z=t$ and $w=q$ recovers the right-hand side in Conjecture~\ref{conj:internal}. 
Despite the similarities, \cite[Remark 1.13]{EHL23} emphasizes that it is unclear how the geometric definitions of $\underline{t}_{\MMM}(x,y,z,w)$ and $t_{\MMM}(x,y,z,w)$ are related. 
This also informally suggests that understanding the case of this subsection could involve techniques different than those employed in this article.

\subsection{The case $k > 1$}
\label{subsec:k_ge_1}

For any arrangement $\AAA$ in $\CC^n$ and 
any integer $k \geq -1$,
Ardila and Postnikov defined 
an ideal $J_{\AAA,k} \subseteq \CC[\xx_n]$
with a generator $\lambda_L^{\rho_\AAA(L) + k}$
for each line $L$ in $\CC^n$.
These ideals have favorable properties for $k \geq -1$.
This paper studied the differential closure
$I_\AAA$ of the ideal $J_{\AAA,1}$.
It is natural to ask for superspace generalizations
of the ideals $J_{\AAA,k}$ for $k > 1$.

For $k > 1$, the most na\"ive generalization
of $J_{\AAA,k}$ is the differential closure
\begin{equation}
I_{\AAA,k} \coloneqq \left( \lambda_L^{\rho_\AAA(L) + k }, d \lambda_L^{\rho_\AAA(L) + k } \,:\, L \subseteq \CC^n \text{ a line}  \right) \subseteq \Omega_n
\end{equation}
of $J_{\AAA,k}$ inside $\Omega_n$.
When $k = 1$, we have $I_{\AAA,1} = I_\AAA$.
The authors have been unable to find a direct connection
between the superspace quotient 
$\Omega_n/I_{\AAA,k}$ and the combinatorics of 
the arrangement $\AAA$ when $k > 1$; some 
difficulties are illustrated in the following example.

\begin{example}
Consider 
the ideal $I_{\AAA,k}$ when $k = 2$ 
and $\AAA$ is the arrangement 
defined by $x_1 x_2 \in \CC[x_1, x_2]$.
The coordinate axes in $\CC^2$ give rise to the generators 
$$
x_1^3, x_2^3, x_1^2 \theta_1, x_2^2 \theta_2
$$
of $I_{\AAA,2} \subseteq \Omega_2$. If $a, b \in \CC^{\times}$ are nonzero complex numbers, the line $\CC \cdot (a,b)$ gives rise to the generators
$$
(a x_1 + b x_2)^4, (a x_1 + b x_2)^3 (a \theta_1 + b \theta_2)
$$
of $I_{\AAA,2}$.  
The quotient ring $\Omega_2 / I_{\AAA,2}$ has vector space basis
\begin{multline*}
\{1\} \cup \{x_1, x_2\} \cup \{x_1^2, x_1 x_2, x_2^2 \} \cup \{ x_1^2 x_2, x_1 x_2^2 \}  \\
\cup \{\theta_1, \theta_2 \} \cup \{ x_1 \theta_1, x_1 \theta_2, x_2 \theta_1, x_2 \theta_2 \} \cup \{ x_1 x_2 \theta_1, x_1 x_2 \theta_2, x_1^2 \theta_2, x_2^2 \theta_1 \} \cup
\{ x_1^2 x_2 \theta_2 \} \\
\cup \{ \theta_1 \theta_2 \} \cup \{ x_1 \theta_1 \theta_2, x_2 \theta_1 \theta_2 \} \cup \{ x_1 x_2 \theta_1 \theta_2 \}
\end{multline*}
where we have grouped monomials according to their bidegrees.  In particular, we have
\begin{equation*}
\Hilb( \Omega_2 / I_{\AAA,2};q,t) = 1 + 2q + 3q^2 + 2 q^3 + 2t + 4 qt + 4 q^2 t + q^3 t + t^2 + 2 q t^2 + q^2 t^2.
\end{equation*}
Setting $q \rightarrow 1$ in this Hilbert series gives 
$8 + 11 t + 4 t^2$. We do not know how to relate this polynomial to the combinatorics of $\AAA$.
\end{example}

We note further that Ardila--Postnikov consider power ideals of two flavors in \cite[Section 4.1]{AP10}, and they happen to coincide in the internal, central, and external cases.
It is currently unclear to us which of those two choices is the `correct' one to generalize to the superspace case when~$k>1$.

\section*{Acknowledgements}
We are extremely grateful to Andy Berget, Sarah Brauner, Yu Li, Vic Reiner, Jos\'e Alejandro Samper, and Hunter Spink for enlightening conversations and/or correspondence.

\bibliographystyle{hplain}
\bibliography{the_bibliography}

\end{document}